\documentclass[11pt]{amsart}
\usepackage{latexsym}
\usepackage{amsmath,amsthm,amsfonts,amssymb}
\usepackage[T1]{fontenc}
\usepackage[latin1]{inputenc}
\usepackage{aeguill} 
\usepackage{url} 
\usepackage{enumerate} 
\usepackage{mdwlist} 

\usepackage{graphicx}
\usepackage{xcolor}
\usepackage[a4paper,textwidth=15cm,textheight=25cm]{geometry}
\usepackage{xspace}
\usepackage{export} 
\usepackage{pgf, tikz}
\usetikzlibrary{decorations.pathreplacing,calligraphy}
\usetikzlibrary{arrows,decorations.markings}
\tikzset{fleche/.style args={#1:#2}{ postaction = decorate,decoration={name=markings,mark=at position #1 with {\arrow{triangle 45}[#2,scale=1]{>}}}}}

    \usepackage[font=small,labelsep=none]{caption}

\newtheorem{thm}{Theorem}[section]

\newtheorem*{thm*}{Theorem}
\newtheorem{dfn}[thm]{Definition} 
\newtheorem*{dfn*}{Definition}

\newtheorem{cor}[thm]{Corollary}
\newtheorem*{cor*}{Corollary}

\newtheorem{prop}[thm]{Proposition} 
\newtheorem*{prop*}{Proposition} 
\newtheorem*{properties*}{Properties} 
 
\newtheorem{lem}[thm]{Lemma} 
\newtheorem*{lem*}{Lemma}

\newtheorem*{claim*}{Claim} 
 
\newtheorem*{fact*}{Fact}

\newtheorem*{qst*}{Question}

\newtheorem*{pb*}{Problem}

\theoremstyle{remark}
 
\newtheorem*{algo*}{Algorithm} 
\newtheorem*{rem*}{Remark}
\newtheorem{rem}[thm]{Remark}
\newtheorem*{example*}{Example}
\newtheorem{example}[thm]{Example}

\newcounter{numEnonceTmpInterne} 
\newenvironment{enonce*}[1]{\theoremstyle{plain}\stepcounter{numEnonceTmpInterne} 
\def\a{enoncetmp\alph{numEnonceTmpInterne}} 
\newtheorem*{\a}{#1}\begin{\a}}{\end{\a}}

\makeatletter
\edef\@tempa#1#2{\def#1{\mathaccent\string"\noexpand\accentclass@#2 }}
\@tempa\rond{017}
\makeatother

\newcommand{\es}{\emptyset}
\renewcommand{\phi}{\varphi} 
\newcommand{\m} {^{-1}}

\newcommand {\cale} {{\mathcal {E}}}

\newcommand {\call} {{\mathcal {L}}}

\newcommand{\Fix}{{\mathrm{Fix}}}

\newcommand{\Out} {{\mathrm{Out}}}

\newcommand{\Aut} {{\mathrm{Aut}}}
\newcommand{\Inn} {{\mathrm{Inn}}}

\newcommand{\pr}{\cp^{rec}}

\newcommand{\cp}{{\sf  p}}
\newcommand{\cq}{{\sf  q}}

\usepackage{srcltx}
\usepackage[all]{xypic}

\newcommand{\Z}{{\mathbb {Z}}}

\usepackage{pdfsync}
\newcommand{\inc}{\subset}

\newcommand {\N} {{\mathbb {N}}} 

\newcommand {\R} {{\mathbb {R}}}

\newcommand {\F} {{\mathbb {F}}}

\newcommand{\bo}{\partial}

\newcommand\mii[4]{\left(\begin{array}{rr} #1&#2\\ #3&#4 \end{array}\right)}

\begin{document}

\title{A Pansiot-type subword complexity theorem for automorphisms of free groups}
\author{Arnaud Hilion and Gilbert Levitt}
\date{\today. 
}

\maketitle

\begin{abstract}
Inspired by Pansiot's work on substitutions, we prove a similar theorem for automorphisms of a free group $F$ of finite rank: if  a right-infinite word $X$ represents an attracting  fixed point of an automorphism of $F$,   the subword complexity of $X$ is equivalent to $n$, $n\log \log n$, $n\log n$, or $n^2$. The proof uses combinatorial arguments analogue to Pansiot's as well as  train tracks. We also define the recurrence complexity of $X$, and we apply it to laminations. In particular, we show that attracting laminations have complexity equivalent to $n$, $n\log \log n$, $n\log n$, or $n^2$ (to $n$ if the automorphism is fully irreducible).
\end{abstract}
  
 \section{Introduction}

   Let $\mathcal A =\{a_1,\dots,a_d\}$ be a finite set, called an alphabet.
 A finite word (on $\mathcal A$) of length $n\geq 1$ is a finite sequence $x_1\dots x_n \in \mathcal A ^n$. 
  The free monoid on $\mathcal A$ is the set $\mathcal A^* = \bigcup_{n\ge0 
 } \mathcal A^n$ consisting of all   finite words together with the empty word $\varepsilon\in \mathcal A^0$ (of length 0), 
 equipped with the concatenation law.
 Elements of $\mathcal A ^\N$ are    (right)-infinite words. Elements of $\mathcal A ^\Z$ are bi-infinite words.

 Given an infinite word $X=x_1\dots x_k \dots$ or a bi-infinite word $X= \dots x_{-1}x_0x_1\dots x_k \dots$ on a finite alphabet $\mathcal A$, its {\em complexity} (or subword complexity, factor complexity) is the function counting,   for each $n\ge1$, the number of factors (or subwords) of length $n$ in X:
 \begin{equation*}\label{eq:complexity}
 \cp_X(n)=\#\{ w\in \mathcal A^n \;|\; \exists\; i,\, \; w=x_ix_{i+1}\dots x_{i+n-1} \}.
 \end{equation*}
Similarly, given a set $E$ of words (if these words are all finite, we require that $E$ be infinite), the complexity of $E$ is:
 \begin{equation*}\label{eq:complexity}
 \cp_E(n)=\#\{ w\in \mathcal A^n \;|\; \exists x\in E, \; \exists i, \; w=x_ix_{i+1}\dots x_{i+n-1} \}.
 \end{equation*}
 
 The complexity is submultiplicative,   in the sense that $\cp_X(n+n')\le \cp_X(n)\cp_X(n')$, so that the limit
 $$h(X)=\lim_{n\rightarrow+\infty}\frac{\log \cp_X(n)}{n\log d}$$
exists: it is the well-known topological entropy of $X$.
In this sense, the complexity of a word $X$ can be seen as a finer dynamical invariant than entropy, which is  particularly relevant in the case of $0$ entropy (i.e.\ when the complexity is subexponential).

Indeed, the interest in complexity has historically focused on low complexity words: the intuition is that a word with low complexity should be ``simple''. 
For instance, the words with bounded complexity are exactly the eventually periodic words. More is true, as shown by Hedlund and Morse \cite{morse} :  $X\in \mathcal A^\N$ is eventually periodic if and only if there exists $n\in\N$ such that  $\cp_X(n)\le n$.
 
Just one step further, we find the Sturmian words. These   have been largely studied (see for instance \cite[chapter 2]{Lothaire}, \cite{Series}) as they appear as the natural coding of lines in the plane $\R^2$ with irrational slope: such a line cuts the entire grid (joining the points of $\Z^2$ by horizontal and vertical lines), giving rise to a bi-infinite word in the letters $h$ (horizontal) and $v$ (vertical). 
In terms of complexity \cite{morse}, a word $X\in \mathcal A^\Z$ is Sturmian if and only if $\cp_X(n)=n+1$ for all $n\in \N$.

This result has motivated an   abundant literature on words of low complexity -- see for instance \cite{Allouche,Cassaigne-Nicolas-in-CANT, Ferenczi} and references therein --
with two kind of approaches: 
on one side determining the kind of functions that can appear as the complexity of a word
-- see for instance \cite{Cassaigne_Kyoto-2002, Ferenczi,  Mauduit-Moreira} -- 
and on the other side developing methods for, given a word (whether arising from    arithmetics, dynamical systems, combinatorics, and so on), computing (or estimating) its complexity
-- see for instance \cite{Cassaigne_Bull-Belg-97}.
 Examples thus play  a central role in these investigations.

In the present paper, we are interested in a result of Pansiot \cite{pansiot} about  
   infinite words obtained by iterating a non-erasing substitution on a letter.
 
A substitution $\sigma$ on a finite alphabet $\mathcal A$ may be viewed as an endomorphism of the free monoid $\mathcal A^*$. Concretely, $\sigma$ is defined by the data $\sigma(a)\in\mathcal A^*$ for all $a\in\mathcal A$. It is non-erasing if $\sigma(a)\neq\varepsilon$ for all $a\in\mathcal A$. 

If $\sigma$ is non-erasing, then, for all $a\in\mathcal A$ such that $\sigma(a)$ has length at least 2 and the first letter of $\sigma(a)$ is $a$, 
the word $\sigma^k(a)$ is a prefix of $\sigma^{k+1}(a)$ and the length of $\sigma^k(a)$   grows to infinity.
Hence the sequence $\sigma^k(a)$ converges to an infinite word $X\in\mathcal A^*$ (of which every $\sigma^k(a)$ is a prefix). 

Pansiot's theorem asserts that the complexity of such an $X$ is bounded, linear (equivalent to $n$), quadratic (equivalent to $n^2$), equivalent to $n\log n$, or equivalent to $n\log \log n$ 
(in this setting, ``$f(n)$ equivalent to $g(n)$'' means that there exist two positive constants $c_1$ and $c_2$ such that  
$c_1g(n) \leq f(n) \leq c_2 g(n)$  for all $n$).
This result strengthens   
a previous result of  Ehrenfeucht, Lee, and Rozenberg \cite{Ehrenfeucht-Lee-Rozenberg} who showed,  in terms of D0L systems,  that the complexity of $X$ is at most quadratic (the total complexity introduced in Definition \ref{totcomp} is similar to the complexity of D0L languages).
We refer to \cite{Cassaigne-Nicolas-in-CANT} for a detailed proof of Pansiot's theorem and a lot of complements.

 The main purpose of this paper is to prove the analogue of Pansiot's theorem for  arbitrary automorphisms of a free  group  $F$ of finite rank (its main results were announced in \cite{turku}). 
The obstruction to  extending the result to injective endomorphisms of free groups is the lack of a theory of train tracks for them (see the work of Mutanguha \cite{Mut} in this direction, however).

 There are several key differences between a free monoid and a free group. In particular, a free group $F$ is much more flexible: a free monoid   has finitely many automorphisms (one may only permute the generators), while the automorphism group $\Aut(F)$ is a big group which is currently under active investigation (see \cite{vogicm} for instance).

 In fact, one can view a free group $F$ as an abstract object rather than a set of words. This is analogous to linear algebra, where one  studies $n$-dimensional vector spaces over $K$ rather than just $K^n$.  With this point of view, our Pansiot-type theorem appears as a theorem about an element $X$ belonging to  the boundary $\bo F$ of $F$   which is fixed under the action of an automorphism.
 
  One may view $\bo F$ as the space of ends of $F$, or its Gromov boundary,    see for instance 
    \cite{BridHa,cooper,  GH,Loh}. 
  Topologically, $\bo F$ is  
homeomorphic to 
a Cantor set, and $F\cup\bo F$ should be viewed as a compactification of $F$ (equipped with a word metric).
 There is a natural   action of $\Aut(F)$ on $F\cup\bo F$ by homeomorphisms   (see \cite{cooper}).

Algebraically, we choose a free basis $A$ of $F$.   Elements of $F$ are reduced  words on the alphabet $E=A\cup A\m$   (reduced means that no letter is followed by its inverse). Elements $X$ of $\bo F$ are represented by reduced right-infinite  words $X_A$. A sequence   in $F\cup\bo F$ converges if, for each integer $s$, the $s$-prefixes converge.
 
   Given $X\in\partial F$, the complexity of the word $X_A$ representing it in the basis $A$ depends on $A$, but in a very controlled way. If $B$ is another basis,   the complexity functions    $\cp_A$, $\cp_B$ of $X_A$, $X_B$ are \emph{equivalent} (denoted $\cp_A\sim \cp_B$)  in the following sense: there exists  a positive integer $C $ such that $\cp_A(n)\le C\cp_B(Cn)$ and  $\cp_B(n)\le C\cp_A(Cn)$ for all $n$ (Proposition   \ref{compeq}). This basic fact may also be expressed as saying that two right-infinite words $X_A$, $Y_A$ which differ by an automorphism $\alpha$ of $F$ (i.e.\ $Y_A=\alpha(X_A)$) have equivalent complexities.
 
  The complexity $\cp_X$ of an element $X\in\bo  F$ is therefore well-defined up to equivalence.   In particular, it makes sense to say that the complexity of $X$ is linear (equivalent to $n$), quadratic (equivalent to $n^2$), equivalent to $n\log n$, equivalent to $n\log \log n$.

   If $\alpha\in\Aut(F)$ and $X\in\bo F$ is the limit of a sequence $\alpha^p(g)$, with $g\in F$, as $p\to+\infty$, as is the case in Pansiot's situation, then $X$ is fixed by $\alpha$.

 Conversely, if $X\in\bo F$ is fixed by $\alpha$, there are several possibilities (see \cite{GJLL}).  First, $X$ may belong to the boundary of the fixed subgroup $\Fix\ \alpha=\{g\in F\mid \alpha(g)=g\}$. This subgroup may have rank 2 or more, and there is no hope of saying anything general about $\cp_X$ in this case. If $X\notin\bo\Fix\ \alpha$, however, it is either attracting or repelling   for the action of $\alpha$ on $F\cup \bo F$ (see \cite{GJLL}). 
  
 We can now state our main result. The various possibilities are illustrated in Section \ref{examp}.

 \begin{thm}[Theorem \ref{main}] \label{mainintro2}
 If $X\in\bo F$ is fixed by $\alpha$ and $X\notin  \bo\Fix\ \alpha$,
 then the  complexity $\cp_X$ of $X$ is   linear, quadratic, equivalent to $n\log n$, or equivalent to $n\log \log n$ (bounded cannot occur).
\end{thm}

\begin{rem}
  In this statement, equivalence may be understood as defined above, or as Lipschitz-equivalence: two functions $f$ and $g$ are Lipschitz-equivalent if there exists $C>0$ such that $\frac 1 C\le \frac{f(x)}{g(x)}\le C$ for all $x$ (see Remark \ref{equiv2}).
\end{rem}

\begin{cor}\label{mainintro}
 Let $\alpha$ be an automorphism of a finitely generated free group $F$, and $g\in F$.    If  $\alpha^p(g)$ converges to an element $X\in \bo F$ as $p\to+\infty$, then the  complexity $\cp_X$ of $X$ is bounded, linear, quadratic, equivalent to $n\log n$, or equivalent to $n\log \log n$.
\end{cor}

One may be more precise, under additional assumptions on $\alpha$ (see Section \ref{cfp} for definitions).
 
\begin{cor}[Corollary \ref {cassimples}] \label{cassimplesintro}  Under the assumptions of   Theorem \ref{mainintro2} or Corollary \ref{mainintro}:

\begin{itemize}
 
\item
If $\alpha$ is fully irreducible, the complexity of $X$  is linear.
\item
If $\alpha$ is atoroidal, the complexity of $X$ cannot be quadratic (it is equivalent  to $n$, $n\log n$, or   $n\log \log n$).

\item
 If $\alpha$ is polynomially growing, the complexity of $X$ is  bounded or quadratic.  
 \end{itemize}
\end{cor}

Let us now give the main ideas of the proof of   Theorem   
\ref{mainintro2}.
 Unfortunately,
 one cannot prove  it  by simply viewing $\alpha$ as a substitution on $A\cup A\m$ and applying Pansiot's theorem, because of \emph{cancellation}. 
 
Indeed, cancellation is another key difference between free monoids and free groups. When applying $\alpha$ to a (finite or infinite) word, one replaces each letter by its image, but one also has to make the word reduced, by successively cancelling all subwords of the form $\beta\beta\m$ or $\beta\m\beta$ with $\beta\in A$ (the order in which the removals are performed is irrelevant).

As a example, 
consider the so-called {\em Tribonacci} substitution on the alphabet $\{a,b,c\}$, given by $\sigma(a)=ab$, $\sigma(b)=ac$, $\sigma(c)=a$. It defines a fully irreducible automorphism $\alpha$ of the free group $F(a,b,c)$,  whose cube fixes the infinite word $$X=\lim_{p\to+ \infty} \alpha^{-3p}(a\m)=a^{-1}ba^{-1}bc^{-1}b a^{-1}ba^{-1}bc^{-1}b c^{-2}bc^{-1}b\dots$$ 
This word is a repelling fixed point of  $\alpha^3$, and a lot of cancellation occurs when $\alpha$ is applied to it.
By Corollary \ref{cassimplesintro}, the complexity of $X$ is linear.

 From a technical point of view, here are two typical exemples (known as INP's and exceptional paths) which do not appear when considering substitutions    and force us to use different techniques. 
 
 Consider the automorphism  $\alpha_1$ of the free group $F(a,b)$ on two generators $a,b$ sending $a$ to $ab$ and $b$ to $bab$. It sends the commutator $[a,b]=aba\m b\m$ to $abbab b\m a\m b\m a\m b\m$, which reduces to $aba\m b\m$: the element $[a,b]$ is fixed. 
 Now consider the automorphism $\alpha_2$ of $F(a,b,c)$ sending $a$ to $a$, $b$ to $ba$, and $c$
 to $ca^2$. Applying powers of $\alpha_2$ to  $bc\m$  sends it to $ba\m c\m$, then to $ba^{-2} c\m$, $ba^{-3} c\m$,...  Note that $ba^{-k}c\m$ is an infinite sequence of words which cannot be written as a concatenation $uv$ such that no cancellation occurs between $\alpha_2(u)$ and $\alpha_2(v)$.
 
   \emph{Exceptional paths} such as $ba^{-k}c\m$, with $k$ unbounded,  are the key obstruction to applying Pansiot's theorem directly. In Subsection \ref{lbd}, exceptional paths will be ruled out and we will be able to use substitutions to shorten the    proof. This technique was first used by R.\ Gupta \cite{Gup} to estimate the frequency of a path $\gamma$. Note, however, that the substitution $\zeta_\gamma$ that she uses depends on  $\gamma$.
 
 The standard way to control cancellation is to use \emph{train tracks}.  
 Train tracks for automorphisms of free groups were introduced by Bestvina-Handel \cite{BHtrain} and improved by Bestvina-Feighn-Handel \cite{Tits1}. We use the \emph{completely split train tracks} constructed by Feighn-Handel \cite{recogn}. See Section \ref{cstt} for definitions, and \cite{Bogop} for examples of train tracks.

 Any automorphism of a free group may be represented by a map $f$ from a graph   to itself. For instance, $\alpha_1$ above is represented on a rose with two petals labelled $a$ and $b$ (see Figure \ref{theta graph} below). The map $f$ wraps the first half of the edge labelled $a$ over   the whole edge and the second half over the edge labelled $b$. It sends the $b$-edge to a path consisting of 3 edges ($a$, then $b$, then $a$).

The idea of  train tracks is to represent a  given   automorphism $\alpha$ by a continuous map $f$ on a graph $G$ with possibly more than one vertex ($f$ represents $\alpha$ in the sense that we can identify the fundamental group of $G$ with $F$ so that  $\alpha$ is the automorphism induced by $f$). The map sends each vertex to a vertex, each edge to a path consisting of one or more edges. The point is to construct   $f$   in such a way that as little cancellation as possible occurs when powers of $f$ are applied to a given edge.

Using \cite{recogn}, we   represent $\alpha$ by  a completely split train track map $f:G\to G$. Any $X\in\bo F$ is represented by a \emph{ray} $\rho$ in $G$: a locally injective map from $[0,\infty)$ to $G$ sending each interval $[p,p+1]$ onto an edge. 
Up to equivalence, the complexity of $X$ may be computed by counting the  complexity of $\rho$, i.e.\  the number of subpaths   of length $n$   (see Lemma \ref{crayon}).

Proposition \ref{comptot} states that the   complexity of $\rho$   is bounded or satisfies the conclusion of Theorem  \ref{mainintro2}. Its proof  combines combinatorial arguments used by Pansiot and geometric arguments using the properties of completely split train tracks. Theorem    \ref{mainintro2} %and  \ref{mainintro2}
 follows fairly directly from Proposition \ref{comptot},    see Section \ref{cfp}.

In Section \ref{reccom} we  also associate to any $X\in\bo F$ a \emph{recurrence complexity} $\pr_X$, by only counting words of length $n$ that appear infinitely often in $X$, and we  show:

\begin{thm}[Theorem \ref{compar3}]
  If $X$ is as in Theorem  \ref{mainintro2} or   Corollary  \ref{mainintro}, its recurrence complexity $\pr_X$ is bounded,  linear, quadratic, equivalent to $n\log n$, or equivalent to $n\log \log n$.  If $\pr_X$ is not quadratic, it is equivalent to $\cp_X$.
\end{thm}

On the other hand,
if
  $\alpha$ is the automorphism of $\F_3$ defined by $a\mapsto a,b\mapsto  ba,c\mapsto  cb$, the fixed point $X=\lim_{p\to\infty}\alpha^p(c)=cbbaba^2\cdots$   has quadratic complexity but linear recurrence complexity  (see Example \ref{rcdif}). 
  
One may associate to any $X\in \bo F$ a \emph{lamination} $L_X$ (see Section \ref{lami}), and $\pr_X$ is the complexity of $L_X$.  
  Theorem  \ref{mainintro2} and   Corollary  \ref{mainintro} thus control the complexity of $L_X$, for $X$ fixed by $\alpha$. This applies in particular to the attracting laminations constructed by Bestvina-Feighn-Handel in \cite{Tits1}.

\begin{thm} [Corollary \ref{corlam}]
 Let $\Phi\in\Out(\F_n)$.
 The complexity of any attracting lamination of   $\Phi$ is linear, quadratic, equivalent to $n\log n$, or equivalent to $n\log \log n$. It is linear if $\Phi$ is fully irreducible, at most $n\log n$ if $\Phi$ is atoroidal.
\end{thm}

 \bigskip

\paragraph{Acknowledgments}
We are grateful to word combinatoricians and group theorists in Marseille for the interest they have expressed in this work, and more specifically to Julien Cassaigne who introduced us to Pansiot's result and spent a lot of time explaining  its proof to us.
We thank the referees for a careful reading and wise suggestions which helped simplify Section 3.

 \section{Preliminaries}

\subsection{Examples} \label{examp}

   We first give examples illustrating the various possibilities in   Theorem \ref{mainintro2}  
   and features of its proof. They may be viewed as positive automorphisms of  $F$, or as invertible substitutions.  We selected them so that they illustrate train track maps and their strata. We refer to Sections \ref{cstt} and \ref{totc}
 for definitions.

We use as a building block the following automorphism $\omega$ of the free group $F(a,b)$ on two generators $a,b$: it sends $a$ to $ab$ and $b$ to $bab$. It is realized geometrically by a pseudo-Anosov homeomorphism of a punctured torus.  The infinite word $X=ab^2ab^2abab\cdots=\lim_{n\to\infty}\omega^n(a)$ is an attracting fixed point whose complexity is linear.  
 
The obvious map representing $\omega$ on a rose with two petals is a completely split train track map. There is one exponential  (EG) stratum, consisting of the whole graph. 
The commutator $[a,b]=aba\m b\m$ is fixed under $\omega$, it is represented by an INP.
Similarly, in the examples given below, the obvious map on a rose is a completely split train track map.

We now consider a free group $F(a,b,c,d)$ of rank 4, which we view as $F(a,b)*F(c,d)$.  The automorphism $\alpha:
{\left\{\begin{array}{rrr} a&\mapsto &{ab}\\ 
b&\mapsto &{bab}\\ 
c&\mapsto &{cd}\\ 
d&\mapsto& {dcd}\\ 
    \end{array}\right.}$ acting on each of the two free factors  as $\omega$ has two unrelated exponential strata (consisting of the $(a,b)$-edges and the $(c,d)$-edges respectively). 
    There is no divergence (in the sense of  Definition \ref{diver}): 
    under iteration of $\omega$, all four letters $a,b,c,d$ grow like $\lambda^n$, with $\lambda$ the largest eigenvalue of the matrix $\mii1112$.  By this we mean that, for $x\in\{a,b,c,d\}$, the length of $\omega^n(x)$ is equivalent to $\lambda^n$ as $n\to+\infty$. Just as above, the fixed points $X=\lim_{n\to\infty}\alpha^n(a)$ and $Y=\lim_{n\to\infty}\alpha^n(c)$ have linear complexity.

We now create divergence by making the image of the $(a,b)$-stratum  go through the $(c,d)$-stratum. Define 
$\alpha_{1,1} :{\left\{\begin{array}{rrr} 
a&\mapsto &{ab}\\ 
b&\mapsto &{b}{\color{red}\ c}\ ab\\ 
c&\mapsto &{cd}\\ 
d&\mapsto& {dcd}\\ 
    \end{array}\right.}  $. Because of the letter $c$ in the image of 
     $b$,
     the two strata interact. The letters $a$ and $b$ grow like $n\lambda^n$, whereas $c$ and $d$ still grow like $\lambda^n$. There is polynomial divergence, and the attracting fixed point $\lim_{n\to\infty}\alpha_{1,1}^n(a)$ has complexity $n\log\log n$  (up to equivalence).
    
    To create exponential divergence,  we use $\omega^2$ in the $(a,b)$-stratum. Define  
$$\alpha_{2,1}:{\left\{\begin{array}{rrr} 
{a}&\mapsto &{abbab}\\ 
{a}&\mapsto &{bab\   {\color{red} c}}\ abbab\\ 
c&\mapsto &{cd}\\ 
d&\mapsto& {dcd}\\ 
    \end{array}\right.}   .$$ Now $a$ and $b$ grow like $(\lambda^2)^n$, and  the complexity of $\lim_{n\to\infty}\alpha_{2,1}^n(a)$ is  $n\log  n$.

All these examples are positive automorphisms. In general, though, automorphisms of $F$ send the generators to words containing both the generators and their inverses. Here is a simple  example, with the fixed point $X=\lim_{n\to\infty}\alpha_0^n(e)=ecd\m fcbab\m a\m d\m fecd\m f\dots$ containing  exceptional paths (such as $cd\m$). 
$$\alpha_0:{\left\{\begin{array}{rrr} a&\mapsto &{ab}\\ 
b&\mapsto &{bab}\\ 
c&\mapsto &{c\ {\color{red}[a,b]}}\\ 
d&\mapsto& {d\ {\color{red}[a,b]^2}}\\
e&\mapsto& {e\ {\color{red}cd\m} f}\\ 
f&\mapsto& {fe\ {\color{red}cd\m} f}\\ 
    \end{array}\right.}$$
    
The complexity of $X$ is quadratic. On the other hand, the fixed point $Y=\lim_{n\to\infty}\alpha_0^n(c)=c[a,b]^\infty=(c[a,b]c\m)^\infty$ belongs to $\bo\Fix\alpha_0$ (because $c[a,b]c\m\in\Fix\alpha_0$) and has bounded complexity.

 \subsection{Graphs}\label{gra}
 Let $G$ be a finite connected graph. 
  Let $E$ be the set of oriented edges, equipped with the fixed-point free involution $e\mapsto \bar e$ reversing the orientation of edges. We write $o(e)$ for the origin of $e$, and $t(e)$ for its terminal point, so that $t(e)=o(\bar e)$. The graph is not assumed to be simplicial: there may be an edge with $o(e)=t(e)$ (a loop), and distinct edges $e,e'$ with $o(e')=o(e)$ and $t(e')=t(e)$. When drawing pictures, we choose a representative for each pair $(e,\bar e)$ and represent it with an arrow and a label (see Figure \ref{theta graph}).
 
 \begin{figure}
\begin{center}
\begin{tikzpicture}

\draw (0,0) node {$\bullet$} ;

\draw (0, 0) .. controls (1, 1) and (2, 1) .. (2, 0);
\draw[fleche=0.999:black]  (0, 0) .. controls (1, -1) and (2, -1) .. (2, 0) node[pos=0.99][right]{b};
\draw (0, 0) .. controls (-1, 1) and (-2, 1) .. (-2, 0);
\draw[fleche=0.999:black] (0, 0) .. controls (-1, -1) and (-2, -1) .. (-2, 0) node[pos=0.99][left]{a};

\draw (6,0) node {$\bullet$} ;
\draw (10,0) node {$\bullet$} ;

\draw[fleche=0.5:black]  (6,0) to[bend left=80] node[pos=0.52][above]{a} (10,0);
\draw[fleche=0.5:black]  (6,0) to[bend left=0] node[pos=0.52][above]{b} (10,0);
\draw[fleche=0.5:black]  (6,0) to[bend right=80] node[pos=0.52][above]{c} (10,0);
\end{tikzpicture}

\end{center}

\caption{\label{theta graph} : A rose with two petals and a theta graph }    

\end{figure}
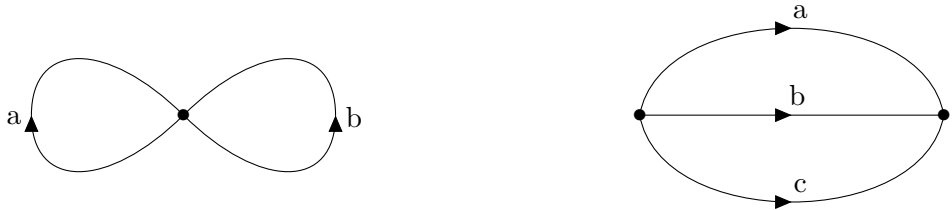

 We assume that the first Betti number $N$ of $G$ (equal to $a_1-a_0+1$ with $a_0$ the number of vertices and $a_1$ the number of non-oriented edges)  is at least 2, so that the fundamental group $ F=\pi_1(G)$, a  free group of rank $N$, is non-abelian.
 
 We view $G$ as a topological space, with universal covering a tree $\tilde G$. One recovers $G$ as the quotient $\tilde G/F$, with $F$   acting  on $\tilde G$ freely by covering transformations. Two distinct vertices of $\tilde G$ are joined by a unique segment, whose length is the number of edges that it contains.  
 
 \subsection{Edge paths and rays}
 \begin{dfn} [Edge path] An edge path in $G$ is a word $\gamma=e_1\dots e_n$ on the alphabet $E$ which is reduced ($e_{i+1}\ne \bar e_i$) and satisfies $o(e_{i+1}) =t(e_i)$ for $i=1,\dots,n-1$.  The inverse path is $\bar\gamma=\bar e_n\dots \bar e_1$ (we sometimes write $\gamma\m$ instead of $\bar\gamma$). 
 
 The length $ |\gamma  | $ of $\gamma$ is the integer $n$. The path is trivial if $n=0$.
 
 A subpath of $\gamma$ is a path of the form $e_re_{r+1} \dots e_s$, with $1\le r\le s\le n$.
 
 For $k\le n$, the $k$-prefix of $\gamma$ is $e_1\dots e_k$, the $k$-suffix is $e_{n-k+1}\dots e_n$. A subpath is interior if it is neither a prefix nor a suffix.
  \end{dfn}
  
   \begin{example} \label{teta}
 Suppose that $G$ is a rose with $N$ petals: it has one vertex and $N$ pairs of  oriented edges $(e_i,\bar e_i)$. An edge path is a reduced word on the alphabet $\{e_1,\bar e_1,\dots,e_N, \bar e_N\}$. If $G$ is a theta graph (see Figure \ref{theta graph}), an edge path is a word on the alphabet $\{a,\bar a, b, \bar b, c,\bar c\}$ in which, for instance,  a letter $a$ may only be followed/preceded by $\bar b$ or $\bar c$ 
 (from a symbolic point of view, the set of  infinite edge paths in this theta graph can be identified with a subshift of finite type of the full shift on the alphabet $\{a,\bar a, b, \bar b, c,\bar c\}$, namely  the subshift consisting of   words  not containing any subword $x\bar x$, $\bar x x$, $xy$, $\bar x\bar y$ for $x,y\in\{a,b,c\}$).
 \end{example}

 \begin{dfn} [Ray] \label{ray}
 A ray $\rho$ in $G$ is a right-infinite word $e_1e_2\dots$ on $E$  satisfying the  conditions $e_{i+1}\ne \bar e_i$ and $o(e_{i+1}) =t(e_i)$ for each $i$.     The origin of $\rho$ is the origin of its initial edge $e_1$.
 
 For $k\ge0$, the \emph{$k$-truncation} $\rho_k$ of $\rho$ is the ray $e_{k+1}e_{k+2}\dots$. 
 
 Two rays $\rho,\rho'$ are \emph{equivalent} if they have a common truncation $\rho_k=\rho'_{k'} $, with $k,k'\ge0$.
 \end{dfn}

We view a non-trivial  edge path $\gamma$    geometrically as a map from $[0,n]$ to $G$ sending $[i-1,i]$ onto $e_i$. This map is usually called an immersion because, $\gamma$ being a reduced word,  it is locally injective. A ray is an immersion from $[0,\infty)$ to $G$. 

An edge path  $\gamma$ may be lifted to an oriented  segment of length $n$ in $\tilde G$. Two oriented segments in $\tilde G$ project to  the same edge path if and only if they differ by the action of an element of $F$. 

When viewed in $\tilde G$, a ray $\rho$ joins a vertex (its origin) to an end of $\tilde G$. Any locally compact space $X$, such as $\tilde G$, has a space of ends $\cale(X)$, defined using complements of compact subsets   (see for instance \cite{BridHa, ross})  The space $\cale(\tilde G)$ is homeomorphic to a Cantor set. More concretely, one may view $\cale(\tilde G)$
as the space of all rays with a fixed origin $\tilde v$. The end of $\rho$ is the only ray $\rho_{\tilde v}$ with origin $\tilde v$ such that $\rho\cap\rho_{\tilde v}$ is infinite. The set of equivalence classes of rays in $G$, as defined above, may be identified with the quotient of $\cale(\tilde G)$ by the natural action of $F$. 

\subsection{Complexity of rays}

\begin{dfn}[Complexity] Let $\rho$ be a ray as in Definition \ref{ray}.
It is a right-infinite word on the alphabet $E$, and we
  let its  (subword, factor) complexity $\cp_\rho$ be the function associating to an integer $n\ge1$ the number of distinct subwords of length $n$.
\end{dfn}

The function $\cp_\rho$ is non-decreasing. It is convenient to   define it for $x$ a positive real number (not just an integer)  by setting $\cp_\rho(x)= \cp_\rho([x])$. It is a standard fact  \cite{morse} 
that $\cp_\rho(n)\ge n+1$ if $\rho$ is not ultimately periodic.

\begin{dfn}[Equivalent functions]\label{eqco}
  Given two non-decreasing functions $\cp $ and $\cq $ defined on $(0,\infty)$, we write $\cp\lesssim \cq$ if there exists $C>0$ such that $\cp(x)\le C\cq(Cx)$   for all $x$.
 We say that $\cp$ and $\cq$  
are equivalent, denoted $\cp\sim \cq$, if $\cp\lesssim \cq$ and $\cq\lesssim \cp$.

\end{dfn}

\begin{rem} [Lipschitz-equivalence]
\label{equiv2}
   When discussing growth, we will consider Lipschitz--equivalence: $\cp$ and $\cq$ are \emph{Lipschitz--equivalent} if there exists $C>0$ such that $\frac1C \cp\le \cq\le C\cp$. 
  Clearly Lipschitz--equivalence implies equivalence in the sense of  Definition \ref{eqco}, though $\lambda^n$ and $\mu^n$ are equivalent but not Lipschitz-equivalent for $\lambda\ne\mu$. On the other hand, the converse (equivalence implies Lipschitz--equivalence) is true if 
$\cp$ is one of the functions $1,n,n^2,n\log n,n\log\log n$, so in most of our statements about complexity equivalence may be understood in either sense.  
\end{rem}

\begin{dfn}[Linear, quadratic]
  We say that $\cp$ is \emph{linear} if it is equivalent to $n$, \emph{quadratic} if it is equivalent to $n^2$. 
\end{dfn}

\begin{example}
 If $\rho$ is a ray, the complexity of a truncation $\rho_k$ is equivalent to the complexity of $\rho$, so equivalent rays have equivalent complexities.
\end{example}

If $e$ is an edge of $G$ with distinct endpoints,   collapsing $e$ to a point yields a graph $G'$ having the same first Betti number, with a projection map 
$\pi:G\to G'$. A ray $\rho$ in $G$ yields a ray $\pi(\rho)$ in $G'$, obtained by erasing the letters $e,\bar e$; if we view $\rho$ as a map from $[0,\infty)$ to $G$, then $\pi(\rho)$ is just a reparametrization of $\pi\circ \rho$. 

\begin{rem} \label{k-sheet}
  Let  $f:G\to G'$ be a $k$-sheeted covering map between finite graphs, for some $k\ge2$. Since any edge path in $G'$ has exactly $k$ lifts to $G$,  the complexity of a ray in $G$ is equivalent to the complexity of its image in $G'$. 
\end{rem}

\begin{lem} \label{cola}
Let $\pi:G\to G'$ be the map obtained by collapsing  an edge $e$ of $G$ with distinct endpoints. If $\rho$ is any ray in $G$,   the complexities of $\pi(\rho)$ and $\rho$ are equivalent.
 \end{lem}

\begin{example}
 Suppose that $G$ is a theta graph, as in Example \ref{teta} (see Figure \ref{theta graph}). Then $\rho$ is a right-infinite word on $\{a,\bar a, b, \bar b, c,\bar c\}$. The lemma claims that erasing all letters $c,\bar c$ does not change the complexity, up to equivalence. 
 \end{example}

\begin{proof}
 The key observation is that, in a ray, no letter $e$ or $\bar e$ may be followed (or preceded) by $e$ or $\bar e$. Any subpath of $\pi(\rho)$ of length $\le n$ is therefore the image of a subpath of $\rho$ of length $\le 2n-1$, and   recalling that the functions are non-decreasing we may write
  $$\frac n2 \cp_{\pi(\rho)}(\frac n2)\le\sum_{m=1}^n \cp_{\pi(\rho)}(m)\le\sum_{m=1}^{2n-1} \cp_\rho(m)\le2n\cp_\rho(2n),
   $$
    (the first inequality is obtained by summing the latter half of the middle term).
   
   In the other direction, the image of a subpath of $\rho$ of length $\le  n$ is a subpath of $\pi(\rho)$ of length $\le  n$, and at most 4 subpaths may have a given projection: if $\pi(\gamma_1)=\pi(\gamma_2)$, one passes from $\gamma_1$ to $\gamma_2$ by inserting or deleting $e$ or $\bar e$ (but not both) at the beginning and/or end of $\gamma_1$. Thus 
    $$\frac n2 \cp_\rho(\frac n2)\le\sum_{m=1}^n \cp_\rho(m)\le4\sum_{m=1}^{ n} \cp_{\pi(\rho)}(m)\le4n\cp_{\pi(\rho)}(n).
   $$
The lemma follows. 
\end{proof}

\subsection{Complexity in free groups} \label{cfg}

A   non-abelian finitely generated free group $F$ has a boundary $\bo F$ (see for instance  
   \cite{BridHa,cooper,  GH,Loh}).  It may be viewed as the space of ends of the group $F$, or as the Gromov boundary of the hyperbolic group $F$ (it is homeomorphic to a Cantor set).  One should view $F\cup\bo F$ as a compactification of $F$, and an element $X$ of $\bo F$ as a limit of elements of $F$.    There are   natural actions of $F$ (by left-translations) and of the automorphism group $\Aut(F)$ on $\bo F$.

Concretely, if we fix a free basis $A=\{a_1,\dots,a_N\}$ of $F$, then  an element  $X \in\bo F$ is just a  reduced right-infinite word  $X_A
$ on the alphabet    $A^{\pm1}$ consisting of the $a_i$'s and their inverses, which we denote by $\bar a_i$. The action of $\alpha\in \Aut(F)$ on $X$ consists in replacing each $a_i$, $\bar a_i$ by its image by $\alpha$ (with $\alpha(\bar a_i)=\alpha(a_i)^{-1}$), and reducing, i.e.\  deleting subwords of the form $a_i\bar a_i$ or $\bar a_i a_i$ (the order in which the deletions are performed is irrelevant).    The conjugation by $g$ acts on $X$ as left multiplication by $g$.

   Geometrically,   choose an isomorphism $\varphi$ between  $F$ and the fundamental group $\pi_1(G,v)$ of a graph $G$ as in Subsection \ref{gra}, with  $v$ a   vertex (such an isomorphism is called a marking of $G$; changing the marking of $G$  is equivalent to applying  an automorphism of $F$).

The boundary of $F$  is then identified to the set of immersed rays in $G$ starting at $v$, and also to   the set of rays in $\tilde G$ starting at a given lift $\tilde v$ of $v$ (the space of ends $\cale(\tilde G)$).    One recovers the previous point of view (right-infinite words on the alphabet $A^{\pm1}$) when $G$ is a rose whose edges represent the elements of a free basis $A$ (see Example \ref{teta}).    In this case, $\tilde G$ is the Cayley graph of $F$ with respect to $A$.

We wish to associate a complexity function to any $X\in\bo F$. 
The obvious idea is to    choose a basis $A$ and define the complexity of $X$ as that of $X_A$, but we have to control how it changes if we use a different basis $B$. Equivalently (considering the automorphism taking $A$ to $B$), we must compare the complexities of $X_A$ and $\alpha(X)_A$ for $\alpha\in\Aut(F)$.
 This is achieved by Proposition~\ref{compeq}, which turns out to be a special case of a general fact stated in Lemma~\ref{crayon} further below.

\begin{prop}\label{compeq}
 Let $A$ and $B$ be two free bases of $F$. 
 
 The complexities of the right-infinite reduced words $X_A$ and $X_B$ representing a given $X\in\bo F$ in $A$ and $B$ respectively are equivalent. 
 
 If $\alpha\in\Aut(F)$, the words $X_A$ and $\alpha(X)_A$ representing $X$ and $\alpha(X)$ respectively have equivalent complexities. 
\end{prop}

Recalling the previous discussion, we get:

\begin{cor} \label{bdef}
 Any element $X\in\bo F$ has a complexity function $\cp_X$ which is well-defined up to equivalence. \qed
\end{cor}

   As explained above,  any $X\in \bo F$ may be represented by a ray with origin $v$ in any   graph $G$ with a marking $\varphi:F\to \pi_1(G,v)$.    The complexity of $X_A$ is the complexity of the ray representing $X$ in the rose associated to $A$.

\begin{lem} \label{crayon}
 Given $X\in \bo F$, the complexity of the ray representing it in $G$ is, up to equivalence, independent of the choices of $G$ and $\varphi$. In particular, it is equivalent to $p_X$.
\end{lem}

\begin{proof}
 This is a   direct consequence of Lemma \ref {cola}. First, using   collapse moves as in the lemma, one may reduce to the case when the graphs are roses. One must then check that the complexity is invariant under a change of the marking (the isomorphism $\varphi$). 
 
   One way of seeing this is by observing that a change of marking is just an element of $\Aut(F)$.    There is a standard   set of generators for $\Aut(F)$, known as elementary Nielsen transformations  (see for instance  \cite{Bogop}, \cite{LS}). Up to renumbering the generators $a_1,\dots,a_N$ of $F$, there are two Nielsen transformations: the first one maps $a_1$ to $a_1\m$ and leaves the other $a_i$'s fixed (so does not change the complexity), the second one maps $a_1$ to $a_1a_2$ and leaves the other $a_i$'s fixed.   A transformation of the second kind may be realized geometrically by applying a collapse move preceded by the inverse of a collapse move (blowing up the vertex of a rose into an edge);  for instance,  collapsing either the  $a$-edge or the $b$-edge   in the theta graph of Figure \ref {theta graph} produces two   roses whose markings differ by a Nielsen automorphism.
   One concludes      
   using Lemma \ref {cola}.
\end{proof}

\begin{rem}\label{restr} 
 If $J$ is a finitely generated subgroup of $F$, its boundary $\bo J$ naturally embeds into $\bo F$. Moreover, by a well-known theorem by M.\ Hall   (see for instance  \cite{LS}, \cite{stal}), $J$ is a free factor of a finite index subgroup $F'\inc F$. 
Let  $F=\pi_1(G)$, and let $f:G'\to G$ be the  covering map associated to the inclusion $F'\inc F$. Remark~\ref{k-sheet}    implies that 
 \emph{given $X\in\bo J$, its complexity in $F$ is equivalent to its complexity in $ F'$, hence to its complexity in $J$.}
\end{rem}
 
\subsection{Train track maps}\label{cstt}

Let $\alpha$ be an automorphism of a finitely generated free group $F$.    It represents an element $\Phi$ of the   group of outer automorphisms $\Out(F)=\Aut(F)/\Inn(F)$, with $\Inn(F)$ the group of inner automorphisms (conjugations).
It is customary to represent   $\alpha$     (or rather $\Phi$)   
by a continuous map $f:G\to G$ (a homotopy equivalence), with  $G$   a finite connected graph whose fundamental group is isomorphic to  $F$. 

Saying that $f$ represents $\alpha$ must be understood as follows. As we do not wish to choose basepoints, a homotopy equivalence of $G$ does not induce an automorphism of $F=\pi_1(G)$, but only an outer automorphism.
We require that the element of $\Out(F)$ induced by $f$ be the image $\Phi$ of $\alpha$   in $\Out(F)$.

We will use \emph {completely split train tracks}, introduced in \cite{recogn} elaborating on \cite{BHtrain} and \cite{Tits1}.   A completely split train track map is a map $f:G\to G$ satisfying many properties.  We describe the properties that we shall need. 

 Unless mentioned otherwise, all definitions and properties below may be found   in \cite{recogn}, where it is proved that any automorphism $\alpha\in\Aut(F)$ has a power $\alpha^k$ which may be represented by a completely split train track map   $f$ on some graph $G$ with $\pi_1(G)\simeq F$ (note that    Theorem \ref{mainintro2}   
  is true for $\alpha$ if they are true for $\alpha^k$). 
   
   If $Z$ is a finite set and $\varphi:Z\to Z$ is any map, there is a power $\psi$ of $\varphi$ such that $\psi(z)$ is fixed by $\psi$ for any $z\in Z$. This observation allows us to assume  that many attributes of   paths are fixed by $f$.

The   train track map $f$ maps a vertex of $G$ to a vertex, an edge to a non-trivial edge path. The restriction of $f$ to an edge path $\gamma$ may fail to be locally injective (because $f(e_{i+1})$ and $f(\bar e_i)$ may start with the same edge), and we define $f_\sharp(\gamma)$ as the edge path obtained by tightening the image:   $f_\sharp(\gamma)$ is homotopic to $f(\gamma)$ relative endpoints and locally injective (there is no backtracking); replacing $f(\gamma)$ by $f_\sharp(\gamma)$ is similar to reducing a word.   If $e$ is an edge, then $f_\sharp(e)=f(e)$ and we usually just write $f(e)$.

A path $\gamma$ is \emph{pre-trivial} if some $f^k_\sharp (\gamma)
 $ is a trivial path. 
 \begin{example}

When $G$ is a rose with $N$ petals, we orient the edges and  label them by letters which we view as generators of $\F_N$. The map $f$ then corresponds to an endomorphism of $\F_N$. 
 For instance, the automorphism of $\F_3$ sending $a$ to $a$, $b$ to $ba$, $c$ to $cb$ is represented by a map $f$ fixing the edge labelled $a$, mapping the $b$ edge to the edge path $ba$, and the $c$ edge to the edge path $cb$. 
\end{example}
 
 All paths in $G$ will be edge paths, so we often drop the word ``edge''.
 
 A \emph{splitting} of an edge path $\gamma$ is a decomposition of $\gamma$ as a concatenation of subpaths $\gamma=\gamma_1\dots \gamma_p$ such that $f^k_\sharp(\gamma)=f^k_\sharp(\gamma_1)\dots f^k_\sharp(\gamma_p)$ for all $k\ge 1$:
  no cancellation occurs between the tightened images of $\gamma_i$ and $\gamma_{i+1}$ when applying a power of $f$. 
 We write $\gamma=\gamma_1\cdot\gamma_2\cdot {\dots} \cdot\gamma_p$ when we want to emphasize that a decomposition of $\gamma$ is a splitting. 
  We define a splitting $\gamma_1\cdot {\dots} \cdot \gamma_p \cdot {\dots}$ of a ray $\rho$ similarly.
 
 A \emph{Nielsen path} (NP) is a non-trivial  edge path $\gamma$ such that  $f_\sharp(\gamma)=\gamma$ (unlike other authors, we do not need to consider paths whose endpoints are not vertices). An edge $e$ is a Nielsen path if and only if it is fixed: $f(e)=e$.  A path is \emph{pre-Nielsen} if some $f^k_\sharp(\gamma)$, with $k\ge0$, is Nielsen. 
 
 There is a filtration $\es=G_0\inc G_1\inc\cdots\inc G_K=G$ representing $G$ as an increasing sequence of (not necessarily connected) subgraphs. The union of the edges contained in $G_r$ but not in $G_{r-1}$ is a subgraph $H_r$ called the \emph{$r$-th stratum.} We say that edges in $H_r$ have \emph{height} $r$. The height $h(\gamma)$ of an edge path is the maximal height of its edges. 
 
 The filtration is compatible with $f$, in the sense that   $f(G_r)\inc G_r$; equivalently,   the image of an edge of height $r$ is an edge path   of height $\le r$.
 
 There are three types of strata. 
 
$\bullet$  $H_r$ is a \emph{0-stratum} if $f(H_r)\inc G_{r-1}$.  All components of a  0-stratum are contractible (they are subtrees of $G$). 
 
$\bullet$  
$H_r$  is an \emph{EG stratum} if, for any edge $e$ of $H_r$, the image $f(e)$ has a splitting $e_1\cdot u \cdot e_2$ with $e_1,e_2$ edges in $H_r$ (the path $u$ may be   trivial); this is not the standard definition, but it is equivalent to it after raising $f$ to a power if needed. No cancellation between edges in $H_r$ occurs when applying $f$ to $f^k_\sharp(e)$, for $k\ge1$ and $e$ an edge of $H_r$ (edges of $H_r$ are  \emph{$r$-legal}).
In particular, the length of $f^k_\sharp(e)$ grows exponentially (EG means exponentially growing). 
 
$\bullet$ The last possibility is 
 an \emph{NEG stratum}: $H_r$ consists of a single edge $e$,    and $f(e)$ is an edge path with a splitting of the form $e\cdot u$ for some 
  loop $u$ contained in $G_{r-1}$. The stratum, and the edge $e$, is \emph{fixed} if $f(e)=e$, \emph{linear} if $u$ is pre-Nielsen   (in particular, not pre-trivial).   
  
An edge is called a 0-edge, an EG edge, an NEG edge depending on the nature of the stratum containing it.

A Nielsen path is \emph{indivisible},  and called an \emph{INP}, if it is not a concatenation of two Nielsen paths. Following tradition,   we do not consider a fixed  edge $e$   as an INP (it becomes divisible if paths whose endpoints are not vertices are allowed). Any Nielsen path is a finite concatenation of INP's and fixed edges. 

If $\gamma$ is an INP of height $r$, the stratum $H_r$ cannot be a 0-stratum. If it is EG, then the initial and terminal edges of $\gamma$ are distinct edges belonging to $H_r$, and (up to changing the orientation) $\gamma$ is the only INP of height $r$.  If $H_r$ is NEG, then   $H_r$  is a linear edge $e$ and  (up to orientation) $\gamma=eu^p\bar e$ with $p\ne0$, where $u$ is a Nielsen path  and $f(e)=eu^d$ with $d>0$. Implicit in the notation $u^d$ is the fact that one can concatenate copies of $u$:    if $u=e_1\dots e_n$, then the terminal point of   $e_n$ is the origin of $e_1$, and  $e_n\ne\bar e_1$; when $d<0$, the notation $u^{d}$   means $\bar u^{-d}$. 

 An \emph{exceptional path} is a path of the form $\gamma=eu^p\bar e'$ where $u$ is a Nielsen path, $p\in\Z$,   
and $e,e'$ are  linear edges such that $f(e)=e\cdot u^d$ and $f(e')=e'\cdot u^{d'}$ with $d,d'>0$; unlike \cite{recogn}, we require 
 $d\ne d'$: if $d=d'$ we view $eu^p\bar e'$ as an INP, 
not as an exceptional path. 
 Note that $f_\sharp(\gamma)$ is an exceptional path if $\gamma$ is one.  

A  non-trivial path $\gamma$ is a \emph{connecting path}  if it is contained in a 0-stratum and its endpoints belong to  an EG stratum $H_r$. A connecting path cannot be pre-trivial. Since 0-strata have contractible components, there is a bound for the length of connecting paths. 

A \emph{complete splitting} of a path $\gamma$ is a splitting $\gamma=\gamma_1\cdot{\dots}\cdot \gamma_p$ such that each $\gamma_i$, called a \emph{term}  (or splitting unit) of the splitting, is one of the following:
\begin{itemize}
\item a single edge in an EG or NEG stratum;
\item an exceptional path;
\item an indivisible Nielsen path (INP);
\item  a connecting path contained in some $f^k_\sharp(e)$, with $k\ge1$ and $e$ an edge in an EG or NEG stratum. 

\end{itemize}

A term is never pre-trivial. There is a uniform bound for the length of terms which are not exceptional or INP's of the form $eu^p\bar e$ with $e$ linear.

A path is \emph{completely split} if it is non-trivial and has a complete splitting.    This complete splitting is unique. We say that its terms are the terms of $\gamma$. A subpath $\gamma_r\dots\gamma_s$, which is a union of   terms of the complete splitting of $\gamma$, will be called a \emph{full subpath} of $\gamma$. It is \emph{interior} if $r>1$ and $s<p$.

 If a subpath of a completely split $\gamma$ is an INP or an exceptional path, it is contained in a term of the complete splitting of $\gamma$ (which is exceptional or an INP), see the proof of Lemma 4.11 in \cite{recogn}.

The image of any   EG or NEG edge is completely split, and its complete splitting is a refinement of the splitting $e_1u e_2$, $eu$ or $\bar u\bar e$ mentioned above. 
The tightened image of a completely split path $\gamma_1\cdot{\dots}\cdot\gamma_p$ is completely split, and its  complete splitting  refines the splitting $f_\sharp(\gamma_1)\cdot{\dots}\cdot f_\sharp(\gamma_p)$. 

If $\gamma$ is an arbitrary edge path, there is an integer $k$ such that $f^k_\sharp(\gamma )$ is completely split.

Generally speaking, many results are proved by induction on height, using the following observations. If $\gamma$ is a connecting path contained in $H_i$, then $f_\sharp(\gamma)\inc G_{i-1}$. If $e$ is an edge in an    NEG stratum $H_i$, then $f(e)$ equals $eu$ or $ue$ with $u\inc G_{i-1}$.

\subsection{Growth 
}
 
 Let $f$ be a completely split train track map as above.
The growth of an edge  path $\gamma$ is the Lipschitz-equivalence class of the function 
$n\mapsto  | f^n_\sharp(\gamma) | $
(see Remark \ref{equiv2}). A path is \emph{growing} if 
$| f^n_\sharp(\gamma) | $
is unbounded (this is equivalent to $\gamma$ not being pre-Nielsen   or pre-trivial). 
Since some $f^k_\sharp(\gamma )$ is completely split, and the growth of a completely split path is the maximal growth of its terms, it suffices to understand the growth of  terms. 

An INP
 and 
a fixed edge do  not grow. An  
exceptional path, and a linear edge as defined above,  grow linearly. If $\gamma$ is a connecting path, $f(\gamma)$ has lower height and one can understand the growth of $\gamma$ by induction. Induction also applies if $e$ is an NEG edge   since  $f^n_\sharp(e)=euf_\sharp(u)\dots f^{n-1}_\sharp(u)$ if $f(e)=eu$; in particular, $e$ has polynomial growth of degree $k+1$ if $u$ has polynomial growth of degree $k$. As previously mentioned, an   EG edge grows exponentially. More precisely:

\begin{prop} \label{growth}
The growth function of any edge    path $\gamma$ is Lipschitz-equivalent to $n^d\lambda^n $ for some $\lambda\ge1$ and some integer $d\ge0$. If $\gamma$ is growing and its growth is polynomial ($\lambda=1$),   
some $f^k_\sharp(\gamma)$ is completely split and there is  a term which is an exceptional path or a linear edge. 
\end{prop}
 
 The first assertion is proved in the appendix of \cite{counting}. The second one follows from remarks made above, using induction on height. 
  We say that $n^d\lambda^n $ is the \emph{growth type} of $\gamma$.
 
  Similarly, if $\alpha\in\Aut(F)$ and $g\in F$, the length of $\alpha^n(g)$ is Lipschitz-equivalent to some $n^d \lambda^n$ as $n\to\infty$   (see the appendix of  \cite{counting}).
 
 \section{Computing the complexity} \label{totc}

  In this section we compute the complexity in a case that can seem specific at first, but we shall explain in Section~\ref{cfp} that we can reduce to this situation, see Lemma~\ref{morceaux}.
 
 Let $f:G\to G$ be a completely split train track map. Let $e$ be an edge such that $f(e)=e\cdot\gamma$ for some completely split path $\gamma$. 
  Then $f^k_\sharp(e)=e\cdot \gamma \cdot {\dots} \cdot f^{k-1}_\sharp(\gamma)$.
 In particular, $f^{k-1}_\sharp(e)$ is the beginning of $f^k_\sharp(e)$, and the increasing union of these paths is
  a completely split ray $\rho=e \cdot \gamma \cdot f_\sharp(\gamma)\cdot {\dots} \cdot f^k_\sharp(\gamma)\cdot {\dots}$ whose complexity we shall compute.  
   For technical reasons, we will also consider the total complexity of $\gamma$, defined as follows.

\begin{dfn}[Total complexity]\label{totcomp}
  We define the \emph{total complexity} $\cp_\gamma(n)$ of a path $\gamma$ as the number of distinct   paths of length $n$ appearing as a subpath of some $f^k_\sharp(\gamma)$   (this is similar to the complexity of D0L languages studied in \cite{Ehrenfeucht-Lee-Rozenberg}).
\end{dfn}

Clearly $\cp_\gamma\le \cp_\rho$ for $\gamma$ and $\rho$ as above. See 
Lemma \ref{compar} for a converse.

 Consider all growing terms in the complete splittings of the paths $f^k_\sharp(\gamma)$, and  their growth   types $n^d\lambda^n $.

\begin{dfn}[Divergence] \label{diver}
 Following \cite{pansiot}, we say that $\gamma$  (or $\rho$) is \emph{non-divergent} if all growing terms have the same $\lambda$ and $d$, \emph{exponentially divergent} if two different $\lambda$'s occur, \emph{polynomially divergent} if all $\lambda$'s are the same but two different $d$'s occur (this is equivalent to saying that either all $\lambda$'s are equal to 1 and  $n^2$ occurs,  or all  $\lambda$'s are equal, different from 1,  and    $n\lambda^n$  occurs). 
 \end{dfn}

 The key result is the following.

\begin{prop} \label{comptot}\label{compquadr}
  
  Let  $e$ with $f(e)=e\cdot\gamma$ and $\rho=e\gamma f_\sharp(\gamma)\cdot{\dots}\cdot f^k_\sharp(\gamma)\dots$ be as above.
 
   \begin{enumerate}
  \item If $\gamma$ does not grow (it is pre-Nielsen),  
the  complexity $\cp_\rho(n)$  is bounded   \emph{(this is obvious).} 
\item If all growing terms of the paths $f^k_\sharp(\gamma)$ grow exponentially, $\cp_\rho(n)$ is equivalent to $n$, $n\log \log n$, or $n\log n$, depending on whether $\gamma$ is non-divergent, polynomially divergent, or exponentially divergent.

  \item  
Otherwise,  
$\cp_\rho(n)$   
   is quadratic (equivalent to $n^2$).
     \end{enumerate} 
  \end{prop}
  
  Equivalence may be understood as Lipschitz-equivalence (see Remark \ref{equiv2}).

\begin{rem}\label{tcf}
The same results hold for  the  {total complexity} $\cp_\gamma(n)$,  
except that in order to get  quadratic complexity we sometimes need $\gamma$ to be ``long enough''  
(see Lemma \ref{entour}). 

For instance, consider the automorphism   of $\F_2$ sending $a$ to $a$ and $b$ to $ba$, and the corresponding map on a rose. The  $b $ edge is linear   but  the total complexity of the path  $bba$ is linear. On the other hand $bbaba^2$, which is of the form $\delta f(\delta) f^2 (\delta)$, does have quadratic total complexity.
\end{rem}

The remainder of this section is devoted to the proof of Proposition \ref{comptot}.  

It is easy to check that   the
proposition  
(as well as most of the statements that we shall make) is true for  
$f$ provided that it is true for some power  
(note that a  power of a completely split train track map is also one). 
We may therefore replace  
$f$ by  a power whenever convenient.   

  We will distinguish between geometric subpaths of $\rho$   
  and abstract paths, i.e.\ words on the alphabet consisting of the edges of $G$; a given abstract path may appear several times in 
  $\rho$. 
  
\subsection{Purely exponential growth} \label{lbd}

  We prove Assertion (2) of Proposition \ref{comptot}.
 By Proposition \ref{growth}, a growing $\gamma$ is as in (2) if and only if no $f^k_\sharp(\gamma)$ has a term which is an exceptional path or a linear edge. Any term appearing in the complete splitting of some $f^k_\sharp(\gamma)$ is either non-growing (an INP, a fixed edge, or a  pre-Nielsen connecting path) or exponentially growing (a single non-linear edge or a connecting path). There are no exceptional paths,
and the length of INP's is bounded, so there is a uniform bound for the length of  terms appearing.
This allows us to  
 introduce a substitution as in  \cite{Gup} and apply  Pansiot's theorem directly.

Let $R$ be the alphabet consisting of $e$ and all terms $\delta$ appearing in the complete splittings of the paths $f^k_\sharp(\gamma)$.  
 The remarks made above imply that $R$ is finite.

The map $f$ induces a substitution $\kappa$ on the alphabet $R$: a term $b$ is sent to the word in $R^*$ given by the complete splitting of $f_\sharp(b)$. Applying powers of $\kappa$ to $e$ produces a $\kappa$-fixed right-infinite word $Y$  in the alphabet $R$ corresponding to the complete splitting of $\rho$.
The complexity of $Y$ is controlled by Pansiot's theorem, we just need to relate it to that of $\rho$.

Let $C$ be a bound for the length (as a path) of all terms in $R$. We can associate to any subword of length $n$ in $Y$ the corresponding completely split subpath of $\rho$. It has length between $n$ and $Cn$. Uniqueness of the complete splitting ensures that this map is injective, so that $$\cp_Y(n)\le\sum_{m=n}^{Cn}\cp_\rho(m).$$

Conversely, consider a subpath $\gamma$ of length $n$ in $\rho$. Let $\hat \gamma$ be the smallest full subpath containing it. This yields a word of length at most $n$ in $R^*$ (different subpaths which are equal as abstract paths may produce different  words, but we do not care). At most  $D$ different paths can produce a given $\hat \gamma$, for some constant $D$ depending on $C$ and the number of edges of $G$, and we get
$$\cp_\rho(n)\le D\sum_{m=1}^{n}\cp_Y(m).$$
  As in the proof of Lemma \ref{cola}, we conclude that $\cp_\rho$ is equivalent to  $\cp_Y$ and we can apply Pansiot's theorem to compute $\cp_\rho$. Moreover, the growth of a term (as a path) is equivalent to  that of the letter of $R$  representing it under $\kappa$, so the divergence used in the proposition  is the same as the divergence that appears in Pansiot's theorem.

This concludes the proof of  the second assertion of Proposition \ref{compquadr}.

 \subsection{Quadratic lower bound} \label{qlb}
 We split the proof of the third assertion of Proposition \ref{comptot} in two: lower bound (in this subsection) and upper bound. 

 We shall deduce a quadratic lower bound for $\cp_\rho(n)$ from
the next lemma. Recall (Definition \ref{totcomp})  that the total complexity $\cp_\delta(n)$ of a path $\delta$  
 is the number of distinct   paths of length $n$ appearing as a subpath of some $f^k_\sharp(\delta)$.

\begin{lem} \label{entour}
The total complexity $\cp_\delta(n)$ is at least quadratic 
if $\delta$ is an edge path with a complete splitting refining a splitting of one of the following forms:
\begin{itemize}
\item  
  $v\theta w$, where $\theta$ is an exceptional path and $v,w$ are growing;
  \item $v\theta w$, where $\theta$ is a linear edge with $f(\theta)=\theta u$,   the paths  $v,w$ are growing, and $w$ has at least two growing terms in its complete splitting.
\end{itemize}
\end{lem}

\begin{proof}
In the first case we write  $\theta=eu^p\bar e'$, with $f(e)=eu^d$ and $f(e')=e'u^{d'}$.
We then have  $f^k_\sharp(\delta)=v_keu^{a_k}\bar e'w_k$, where $ | eu^{a_k}\bar e' | $ is an affine function $ak+b$  with $a>0$ and   $ | v_k | , | w_k | $  grow at least like $ck$ for some $c>0$. Given $n$, consider, for all $k$ such that $n\ge ak+b\ge n/2$ and $ ak+b+cn \ge n$, 
the subpaths $\eta$ of length $n$ of $f^k_\sharp(\delta)$ containing the subpath $eu^{a_k}\bar e'$. Assuming, as we may, $c<\frac12$, and neglecting $b$, we see that $ak$ varies between $n-cn$ and $n$.

 We claim that the number of subpaths thus obtained is equivalent to $n^2$.    
 Indeed, since $ | v_k | , | w_k | $ grow at least like $ck$, there is enough space on either side of $eu^{a_k}\bar e'$  to construct, for each $k$, at least $n-ak$ subwords of $f^k_\sharp(\delta)$ containing $eu^{a_k}\bar e'$.  We thus get  a lower bound $\displaystyle\sum_{   {n-cn}\le ak\le{  n}}(n-ak) $, which is quadratic in $n$.

 Moreover,
the paths that we constructed are all distinct as abstract paths: $\eta$ determines $k$  because the subpath $u^{a_k}$ may be characterized as the only maximal subpath   of $\eta$ whose height   is less than $h(e)$ and $h(e')$, and whose length is at least $n/2$.   

 In the second case, we write $f^k_\sharp(\delta)=v_k\theta u^{ k} w_k$. We now consider $k$ such that $k | u | \ge n/2$ and subpaths $\eta$ of length $n$ containing $\theta u^k$. To show that they are distinct, we consider the maximal subpath   of $\eta$ whose height   is less than $h(\theta)$ and whose length is at least $n/2$. There is an added difficulty 
 because $u$ may be a prefix of $w_k$. 
 
 Write $w=xyz$ with $x$ a (possibly empty)  
   Nielsen path, $y$ a  growing term, and $z$ growing. Using induction on height, we may assume that $y$ is not a connecting path.  
    If it is an EG edge, or an exceptional path, or an   NEG edge with $f(y)=y\tau$, 
     every prefix of $w_k$ of the form $u^p$ is contained in $x$, 
 so distinctness holds.

 The only case remaining is when $y$ is an NEG edge with $f(y)=\tau y$. Using induction, one may assume that  $\tau$ is a Nielsen path, so $f^k_\sharp(\delta)=v_k\theta u^{ k} x\tau^kyz_k$. Distinctness holds 
 unless $x\tau^k$ is a power of $u$ followed by a path of length less than $ | u | $. If this happens, we consider, for $k$ such that $k (| u |+| \tau | )\ge n/2$, the subpaths $\eta$ of length $n$ of $f^k_\sharp(\delta)$ containing $u^{ k} x\tau^k$. 
  As in the first case,  
 the number of these paths is equivalent to $n^2$ because $| v_k |$ and $|z_k|$   grow at least  linearly.
\end{proof}

Suppose that $\gamma$ is as in Proposition \ref{comptot}, and neither 1 nor 2 applies. By 
  Proposition \ref{growth}, there is $k_0$ such that, for $k\ge k_0$, the path   $f^k_\sharp(\gamma)$ contains a subpath $\theta_k$ which is an exceptional path   or a linear edge. We apply the lemma to $\theta=\theta_{k_0+2}$ (with the obvious modifications if $f(\theta)$ equals $u\theta$ rather than $\theta u$)  to get a quadratic lower bound for $p_\gamma(n)$, hence for 
  $p_\rho(n)$.

\subsection{Quadratic upper bound}

 There remains to establish a quadratic upper bound 
for the   complexity $\cp_\rho(n)$, with $\rho$ as in Proposition \ref{compquadr}. We start with a lemma.

  \begin{lem} [Stabilization of prefixes] \label{stabpref}
  Let $f:G\to G$ be a completely split train track map, and $\gamma$   an arbitrary edge path.
There exists a power $g$ of $f$ such that, given any edge path $\gamma$, there is $C>0$ such that, for any $n$,  the $n$-prefix of $g^k_\sharp(\gamma)$, viewed as an abstract path,  is independent of $k$ for $k>Cn$.
\end{lem}

\begin{proof}
 After replacing $\gamma$ by some $f^q_\sharp(\gamma)$, we may assume that $\gamma$ is completely split.
There is a uniform linear lower bound for the length of $f^k_\sharp(\tau)$, for $\tau$ a growing term, so we
may assume that $\gamma$ is a  single growing term. 
Choose $g$ such that $g(e)$ and $g^2_{\sharp}(e)$ have the same initial edge $a_e$  if $e$ is any   EG edge.

The result is clear  if  $\gamma$ is an exceptional path, and also if it is an  NEG edge $e$ such that  $f(e)=eu$ because $g^k_\sharp(\gamma)$ is a prefix of $g^{k+1}_\sharp(\gamma)$. 
If $\gamma$ is an    EG edge  $e$, then $g^k_\sharp(\gamma)$ starts with $g^{k-1}_\sharp(a_e)$; in this case, stabilization occurs for $k=O(\log n)$.
If $\gamma$ is an NEG edge $e$ with $f(e)=ue$, the result is clear if $u$ is pre-Nielsen, and follows by induction on height if $u$ is growing. We  also use induction if $\gamma$ is a connecting path. 
\end{proof}

Replacing $f$ by a power, we may assume that Lemma \ref{stabpref} applies.
We fix a number $K$ which is large compared to   $| \gamma | $   and the length of terms which are not exceptional paths or INP's.

  By abuse, we view $f_\sharp$ as a map from $\rho$ to $\rho$ sending $e$ to $e\gamma$ and $f^k_\sharp(\gamma)$ to $f^{k+1}_\sharp(\gamma)$. 
Consider $n>K$ and  a subpath $w_0$   of length $n$ in $\rho$. Let $w_1$ be the smallest full subpath of $\rho$   such that $f_\sharp(w_1)$ contains $w_0$. Then define $w_2$
 similarly and keep going until the first $p$ for which    $ | w_{p+1} | \le K< | w_p | $ (see Figure   \ref{upper bound}). Such a $p$ exists because 
 $n>K> | \gamma | $.

\begin{figure}
\begin{center}
\begin{tikzpicture}

\draw [thin] (-7,0)--(7,0);
\draw [very thick] (-4,0)--(4,0);
\draw (-4,0) node {$\bullet$} ;
\draw (4,0) node {$\bullet$} ;

\draw (-6,0) node[fill,circle,inner sep=0pt,minimum size=2pt] {} ; 
\draw (-5.6,0) node[fill,circle,inner sep=0pt,minimum size=2pt] {} ; 
\draw (-5.2,0) node[fill,circle,inner sep=0pt,minimum size=2pt] {} ; 
\draw (-4.8,0) node[fill,circle,inner sep=0pt,minimum size=2pt] {} ; 
\draw (-4.4,0) node[fill,circle,inner sep=0pt,minimum size=2pt] {} ;  
\draw (-3.6,0) node[fill,circle,inner sep=0pt,minimum size=3pt] {} ; 
\draw (-3.2,0) node[fill,circle,inner sep=0pt,minimum size=3pt] {} ; 
\draw (-2.8,0) node[fill,circle,inner sep=0pt,minimum size=3pt] {} ; 
\draw (-2.4,0) node[fill,circle,inner sep=0pt,minimum size=3pt] {} ;  
\draw (-2,0) node[fill,circle,inner sep=0pt,minimum size=3pt] {} ;  
\draw (-1.6,0) node[fill,circle,inner sep=0pt,minimum size=3pt] {} ; 
\draw (-1.2,0) node[fill,circle,inner sep=0pt,minimum size=3pt] {} ; 
\draw (-0.8,0) node[fill,circle,inner sep=0pt,minimum size=3pt] {} ; 
\draw (-0.4,0) node[fill,circle,inner sep=0pt,minimum size=3pt] {} ;  
\draw (0,0) node[fill,circle,inner sep=0pt,minimum size=3pt] {} ;  
\draw (0.4,0) node[fill,circle,inner sep=0pt,minimum size=3pt] {} ;  
\draw (0.8,0) node[fill,circle,inner sep=0pt,minimum size=3pt] {} ;  
\draw (1.2,0) node[fill,circle,inner sep=0pt,minimum size=3pt] {} ;  
\draw (1.6,0) node[fill,circle,inner sep=0pt,minimum size=3pt] {} ;  
\draw (2,0) node[fill,circle,inner sep=0pt,minimum size=3pt] {} ;  
\draw (2.4,0) node[fill,circle,inner sep=0pt,minimum size=3pt] {} ;  
\draw (2.8,0) node[fill,circle,inner sep=0pt,minimum size=3pt] {} ;  
\draw (3.2,0) node[fill,circle,inner sep=0pt,minimum size=3pt] {} ;  
\draw (3.6,0) node[fill,circle,inner sep=0pt,minimum size=3pt] {} ;  
\draw (4.4,0) node[fill,circle,inner sep=0pt,minimum size=2pt] {} ;  
\draw (4.8,0) node[fill,circle,inner sep=0pt,minimum size=2pt] {} ;  
\draw (5.2,0) node[fill,circle,inner sep=0pt,minimum size=2pt] {} ;  
\draw (5.6,0) node[fill,circle,inner sep=0pt,minimum size=2pt] {} ;  
\draw (6,0) node[fill,circle,inner sep=0pt,minimum size=2pt] {} ;  
\draw (6.4,0) node[fill,circle,inner sep=0pt,minimum size=2pt] {} ;

\draw [thin] (-7,1)--(7,1);
\draw [very thick] (-2,1)--(2.4,1);
\draw (-2,1) node {$\bullet$} ;
\draw (2.4,1) node {$\bullet$} ;

\draw (-3.2,1) node[fill,circle,inner sep=0pt,minimum size=2pt] {} ; 
\draw (-2.8,1) node[fill,circle,inner sep=0pt,minimum size=2pt] {} ; 
\draw (-2.4,1) node[fill,circle,inner sep=0pt,minimum size=2pt] {} ; 
\draw (-1.6,1) node[fill,circle,inner sep=0pt,minimum size=3pt] {} ; 
\draw (-1.2,1) node[fill,circle,inner sep=0pt,minimum size=3pt] {} ; 
\draw (-0.8,1) node[fill,circle,inner sep=0pt,minimum size=3pt] {} ; 
\draw (-0.4,1) node[fill,circle,inner sep=0pt,minimum size=3pt] {} ; 
\draw (0,1) node[fill,circle,inner sep=0pt,minimum size=3pt] {} ;  
\draw (0.4,1) node[fill,circle,inner sep=0pt,minimum size=3pt] {} ;  
\draw (0.8,1) node[fill,circle,inner sep=0pt,minimum size=3pt] {} ;  
\draw (1.2,1) node[fill,circle,inner sep=0pt,minimum size=3pt] {} ;  
\draw (1.6,1) node[fill,circle,inner sep=0pt,minimum size=3pt] {} ;  
\draw (2,1) node[fill,circle,inner sep=0pt,minimum size=3pt] {} ; 
\draw (2.8,1) node[fill,circle,inner sep=0pt,minimum size=2pt] {} ; 
\draw (3.2,1) node[fill,circle,inner sep=0pt,minimum size=2pt] {} ; 
\draw (3.6,1) node[fill,circle,inner sep=0pt,minimum size=2pt] {} ; 
\draw (4,1) node[fill,circle,inner sep=0pt,minimum size=2pt] {} ;

\draw [very thin, densely dashed] (-2,1)--(-4.8,0);
\draw [very thin, densely dashed] (2.4,1)--(5.2,0);
\draw [very thin, densely dashed] (-1.6,1)--(-2.4,0);
\draw [very thin, densely dashed] (2,1)--(3.2,0);

\draw [thin] (-7,2)--(7,2);
\draw [very thick] (-0.8,2)--(1.2,2);
\draw (-0.8,2) node {$\bullet$} ;
\draw (1.2,2) node {$\bullet$} ;

\draw (-1.6,2) node[fill,circle,inner sep=0pt,minimum size=2pt] {} ; 
\draw (-1.2,2) node[fill,circle,inner sep=0pt,minimum size=2pt] {} ; 
\draw (-0.4,2) node[fill,circle,inner sep=0pt,minimum size=3pt] {} ; 
\draw (0,2) node[fill,circle,inner sep=0pt,minimum size=3pt] {} ; 
\draw (0.4,2) node[fill,circle,inner sep=0pt,minimum size=3pt] {} ; 
\draw (0.8,2) node[fill,circle,inner sep=0pt,minimum size=3pt] {} ; 
\draw (1.6,2) node[fill,circle,inner sep=0pt,minimum size=2pt] {} ; 
\draw (2,2) node[fill,circle,inner sep=0pt,minimum size=2pt] {} ;

\draw [very thin, densely dashed] (-1.2,2)--(-0.7,2.5);
\draw [very thin, densely dashed] (1.6,2)--(1,2.5);

\draw [very thin, densely dashed] (-0.8,2)--(-2.4,1);
\draw [very thin, densely dashed] (1.2,2)--(3.6,1);
\draw [very thin, densely dashed] (-0.4,2)--(-1.2,1);
\draw [very thin, densely dashed] (0.8,2)--(1.6,1);

\draw [very thick, dotted] (0.2,2.5)--(0.2,3.2);

\draw [thin] (-7,4)--(7,4);
\draw [very thick] (-0.4,4)--(0.8,4);
\draw (-0.4,4) node {$\bullet$} ;
\draw (0.8,4) node {$\bullet$} ;

\draw (-1.6,4) node[fill,circle,inner sep=0pt,minimum size=2pt] {} ; 
\draw (-1.2,4) node[fill,circle,inner sep=0pt,minimum size=2pt] {} ; 
\draw (-0.8,4) node[fill,circle,inner sep=0pt,minimum size=2pt] {} ; 
\draw (0,4) node[fill,circle,inner sep=0pt,minimum size=3pt] {} ; 
\draw (0.4,4) node[fill,circle,inner sep=0pt,minimum size=3pt] {} ; 
\draw (1.2,4) node[fill,circle,inner sep=0pt,minimum size=2pt] {} ; 
\draw (1.6,4) node[fill,circle,inner sep=0pt,minimum size=2pt] {} ; 
\draw (2,4) node[fill,circle,inner sep=0pt,minimum size=2pt] {} ;

\draw [very thin, densely dashed] (-0.4,4)--(-0.9,3.5);
\draw [very thin, densely dashed] (0.8,4)--(1.5,3.5);

\draw [thin] (-7,5)--(7,5);
\draw [very thick] (0,5)--(0.4,5);
\draw (0,5) node {$\bullet$} ;
\draw (0.4,5) node {$\bullet$} ;

\draw (-0.8,5) node[fill,circle,inner sep=0pt,minimum size=2pt] {} ; 
\draw (-0.4,5) node[fill,circle,inner sep=0pt,minimum size=2pt] {} ; 
\draw (0.8,5) node[fill,circle,inner sep=0pt,minimum size=2pt] {} ; 
\draw (1.2,5) node[fill,circle,inner sep=0pt,minimum size=2pt] {} ;

\draw [very thin, densely dashed] (0,5)--(-0.8,4);
\draw [very thin, densely dashed] (0.4,5)--(1.6,4);

\draw [decorate, decoration = {calligraphic brace, mirror}] (-4,-0.2) --  (4,-0.2);
\draw (0,-0.5) node {$w_0$} ;
\draw [decorate, decoration = {calligraphic brace, mirror}] (-2,0.8) --  (2.4,0.8);
\draw (0.2,0.5) node {$w_1$} ;
\draw [decorate, decoration = {calligraphic brace, mirror}] (-0.8,1.8) --  (1.2,1.8);
\draw (0.2,1.5) node {$w_2$} ;
\draw [decorate, decoration = {calligraphic brace, mirror}] (-0.4,3.8) --  (0.8,3.8);
\draw (0.2,3.5) node {$w_{p-1}$} ;
\draw [decorate, decoration = {calligraphic brace, mirror}] (0,4.8) --  (0.4,4.8);
\draw (0.2,4.5) node {$w_p$} ;

\end{tikzpicture}

\end{center}

    \caption[Entry for the List of Figures (LoF)]{}\label{upper bound}

\end{figure}

 We  consider the complete splitting of $w_p$. 
 There are two cases.  
 Both may occur for subpaths $w_0$ representing the same abstract path,  
 but in order to bound $\cp_\rho$ up to equivalence we may  choose one for every abstract $w_0$.
 
  First suppose that $w_p$ has  an interior  growing term $\tau$, and call $x$ its origin. We then have $f^{p}_\sharp(\tau)\inc w_0$.   
 Since $\tau$ grows at least linearly and $ | w_0 | =n$, we get 
an inequality $p\le Cn$ with $C$ depending only on $f,\gamma$ and $K$  because $ | w_p | $ is bounded.

  Recall that we want an upper bound for $\cp_\rho(n)$, the number of distinct abstract paths $w_0$ of length $n$ contained in    
 $\rho$.  
  Such a path is determined by the abstract path $w_p$, the number $p$, and the position of $f^p(x)$ within $w_0$. There are boundedly many possibilities for $w_p$ and $x$.  Since there  are $n+1$ possibilities for the position of  $f^p(x)$, and $p\le Cn$, we get the required quadratic upper bound for   $\cp_\rho$.

   The second case is when the complete splitting of $w_p$ has no interior growing term. 
 If $w_p$ is a single term, it must be an exceptional path (a trivial case), so we can assume that $w_p$ may be split   at a vertex $x$ as a concatenation $w_p=vv'$  (one of $v,v'$ may be non-growing).
 As in the first case we have $f^p(x)\in w_0$. The path $w_0$ consists of a suffix of  $f^k_\sharp(v)$ and a prefix of $f^k_\sharp(v')$,
 and the quadratic bound for $\cp_\rho$ follows from stabilization of prefixes (Lemma \ref{stabpref}).

\section{Complexity of  fixed points on the boundary}
\label{cfp}

Let $\alpha$ be an automorphism of a finitely generated free group $F$. As explained in Subsection \ref{cfg}, it acts on the boundary $\bo F$, and we let $X\in\bo F$ be a fixed point. We are interested in the complexity $\cp_X$ of $X$, which is well-defined up to equivalence by Corollary \ref{bdef}.

There are three possibilities for $X$ (see \cite{GJLL}). 
First, $\alpha$ may have a non-trivial fixed subgroup $\Fix \alpha =\{g\in F\mid  \alpha(g)=g\}$, and $X$ may be a point of the boundary of $\Fix \alpha$ (the most trivial example being when $\alpha$ is the identity); in this case $X$ may be fairly arbitrary, and there is no hope  of describing its complexity. 
By Proposition 1.1 of \cite{GJLL}, fixed points $X$ which do not  belong to  $ \bo \Fix\alpha$  are   either  attracting or   repelling (attracting for $\alpha\m$)    for the action of $\alpha$ on $F\cup\bo F$.  

\begin{rem} \label {pasper}
If   $X\in\bo F$ is  rational, i.e.\ of the form $g^\infty=\lim_{n\to+\infty}g^n$ for some  $g\in F$, then $\alpha(X)=\alpha(g)^\infty$, so  $\alpha(X)=X$   implies $\alpha(g)=g$ and $X\in \bo \Fix\alpha$. Thus 
an attracting or repelling fixed point  $X$ cannot be rational:
 viewed as  a right-infinite word,  it  is not ultimately periodic. In particular, its complexity is at least linear by      \cite{morse}.
 \end{rem}
 
\begin{thm} \label{main}
Let $\alpha$ be an automorphism of a finitely generated free group $F$, and let $X\in\bo F$ be fixed by $\alpha$. If $X\notin\bo\Fix\alpha$ (i.e.\ $X$ is attracting or repelling), its complexity $\cp_X$ is equivalent to $n,n\log\log n, n\log n$, or $n^2$.
\end{thm}

   For a given $\alpha$, there are only finitely many attracting/repelling fixed points  
up to the action of $ \Fix\alpha$  by \cite{cooper}. On the other hand, if we fix $\Phi\in\Out(F)$ and vary the representative $\alpha$, the theorem applies to infinitely many $X$ (see Example \ref{lamiex}).

   Theorem \ref{main}  is a rewording of Theorem \ref{mainintro2}. It implies   Corollary  \ref{mainintro} because any $X=\lim_{p\to\infty}\alpha^p(g)$ must be rational or attracting (\cite{these}, Th\'eor\`eme 2.15). 
   It is an immediate consequence of  Proposition \ref{comptot} and the following lemma. 

\begin{lem} \label{morceaux}
Let $X$ be as in 
Theorem~\ref{main}.
There exist a   completely split train track map $f:G\to G$ and a ray $\rho$ in $G$ such that:
\begin{itemize}
 \item the complexity of $\rho$ is equivalent to the complexity of $X$;

 \item 
 there is an edge $e$ such that $f(e)=e\cdot\sigma$ with $\sigma$ a growing completely split path, and $\rho=e\cdot\sigma \cdot f_\sharp (\sigma)\cdot f^2_\sharp(\sigma)\cdot f^3_\sharp(\sigma)\cdots$
\end{itemize}
 \end{lem}

 Unfortunately, it is not true that any $X$ may be represented by a ray $\rho$ as in the lemma. The shortest trick to reduce to this case is to use results of \cite{recogn} about principal automorphisms. We do this at the cost of enlarging $F$.

 \begin{proof}
 This relies on Lemma 4.36 (2) of \cite{recogn}. To apply it, we need $X$ to be attracting and $\alpha$ to be principal (see Definition 3.1 of \cite{recogn}), so we   start by modifying $\alpha$. To make $X$ attracting,  we   replace $\alpha$   by $\alpha{\m}$ if needed. To make $\alpha$ principal, we replace it by the automorphism of $F*\F_2$ equal to $\alpha$ on $F$ and the identity on $\F_2$. We then   raise it to a power, so that it is represented by a   completely split train track map $f:G\to G$ (Theorem 4.28 of \cite{recogn}). We view $X$ as an attracting fixed point of the modified $\alpha$.

 We now apply Lemma 4.36 (2) of \cite{recogn}. It provides a vertex $x$ of $G$ with $f(x)=x$ and a non-linear edge $e$ with origin $x$ such that $e\inc f_\sharp(e)\inc f^2_\sharp(e)\inc\cdots$ is an increasing sequence of paths whose union is a ray $\rho$ representing $X$. We   let $\sigma$ be  the path such that   $f(e)=e\sigma$.
 \end{proof}
 
Recall    (see for instance \cite{Kapo}, \cite{counting}) that $\alpha$ is 
\emph{fully irreducible} (also known as iwip) if  no free factor of $F$ other than $\{1\}$ or $ F$ is mapped to a conjugate by a power of $\alpha$, \emph{atoroidal} if there is no non-trivial $\alpha$-periodic conjugacy class in $F$, \emph{polynomially growing} if    every $g\in F$ grows polynomially under iteration of $\alpha$ (equivalently, 
every edge path grows polynomially under iteration of $f$).

\begin{cor}\label{cassimples}  Let  $X$ be an attracting or repelling  fixed point  of $\alpha$  in $\bo F\setminus\bo Fix \alpha$.
\begin{itemize}
 
\item
If $\alpha$ is fully irreducible, the complexity of $X$  is linear.
\item
If $\alpha$ is atoroidal, the complexity of $X$  is $O(n\log n)$.

\item
 If $\alpha$ is polynomially growing, the complexity of $X$ is quadratic. 
 \end{itemize}
\end{cor}

\begin{proof}
This follows from   
Proposition \ref{comptot}. If $\alpha$ is fully irreducible, there is no divergence   ($G$ consists of a single EG stratum). If $\alpha$ is atoroidal, there is no linear edge or exceptional path    $\gamma$ (as the path $u$ appearing in the definition of $\gamma$ given in Subsection \ref{cstt} would define a fixed conjugacy class), so $\cp_X$ cannot be quadratic. If $\alpha$ is polynomially growing, we must be in the third case of the proposition.
\end{proof}

\begin{rem}

Using  Remark \ref{restr}, one can show that the complexity is quadratic also if $\alpha$ is arbitrary but $X$ belongs to the boundary of a polynomial subgroup (i.e.\ a subgroup $J$ which is invariant under an automorphism $\beta$ representing the same outer automorphism as $\alpha$, 
with $\beta_{ | J}$  polynomially growing, see Proposition 1.4 of \cite{counting}).
\end{rem}

\section{Recurrence complexity} \label{reccom}

Let $f:G\to G$ be a completely split train track map.
 
 \subsection{Preliminaries}
The complexity function $\cp_\rho$ of a ray is defined by counting the abstract paths appearing at least once in $\rho$. We also defined the (total) complexity $\cp_\gamma$ of a path $\gamma$ by counting the   paths appearing in some $f^p_\sharp(\gamma)$.
We now define the recurrence complexity function $\pr$ by counting the paths appearing infinitely often.

\begin{dfn}[Recurrence complexity]
Given a ray $\rho$ in $G$, we define the recurrence complexity function $\pr_\rho(n)$ as the   number of abstract subpaths of length $n$ appearing infinitely often in $\rho$.

Given an edge path $\gamma$, we define $\pr_\gamma(n)$ as the   number of abstract subpaths of length $n$ appearing in infinitely many $f^k_\sharp(\gamma)$.

  Lemma \ref{cola} is true for recurrence complexity, so that any element $X\in\bo F$ has a recurrence complexity $\pr_X$, well-defined up to equivalence.

\end{dfn}

\begin{example}\label{rcdif}
The recurrence complexity may be smaller than the complexity: if
  $\alpha$ is the automorphism of $\F_3$ (or invertible substitution on 3 letters) defined by $a\mapsto a,b\mapsto  ba,c\mapsto  cb$, the fixed point $X=\lim_{p\to\infty}\alpha^p(c)=cbbaba^2\cdots$ and the path $\gamma=bbaba^2$ have quadratic complexity and linear recurrence complexity    (subwords appearing infinitely often contain $b$ at most once).
\end{example}

   \begin{lem} \label{compar}
   Let   
   $\rho=e\cdot\sigma \cdot f_\sharp (\sigma)\cdot f^2_\sharp(\sigma)\cdot f^3_\sharp(\sigma)\cdots $
   be as in Lemma \ref{morceaux}, and let 
    $\gamma=\sigma \cdot f_\sharp (\sigma)\cdot f^2_\sharp(\sigma)$.
   The functions $p_X$ and $p_\gamma$ are equivalent, and  
$\cp_\gamma\sim \cp_X\ge \pr_X\sim\pr_\gamma\gtrsim n$.  
\end{lem}

\begin{proof}
Clearly $\cp_\gamma\le \cp_X$.  To prove  $\cp_X\lesssim \cp_\gamma$, we have to control the subpaths $\tau$ of $\rho$ which are not contained in any $f^k_\sharp(\gamma)$. 
Note that any such $\tau$ contains some $f^\ell_\sharp(\sigma)$, hence has to be long if it is far from the origin of $\rho$.  More precisely, there is a function $g $, depending on the growth type of $\sigma$, such that   $\tau$  
is contained in the $g( | \tau | )$-prefix of $\rho$, and we have $\cp_X(n)\le \cp_\gamma(n)+O(g(n))$.

To compute $g$, we have to estimate the distance from the subpath $f^\ell_\sharp(\sigma)$ to the origin of $\rho$  in terms of its length, in other words to 
bound $\sum_{k=1}^\ell | f^ k_\sharp(\sigma) | $ in terms of $ | f^\ell_\sharp(\sigma) | $.

Let $k^d \lambda^k $ be the growth type of $\sigma$. If $\lambda>1$, then   $\sum_{k=1}^\ell k^d\lambda^k  =O(  \ell^d \lambda^\ell)$ and $g$ is linear, so $\cp_X(n)\le \cp_\gamma(n)+O(n)$, hence $\cp_X\lesssim \cp_\gamma$.  If $\lambda=1$, then   $\sum_{k=1}^\ell   k^d =O(    \ell^{d+1})$ and $g(n)\sim n^{\frac{d+1}d}$. We deduce $\cp_X(n)\le \cp_\gamma(n)+O(n^{\frac{d+1}d})$, and $\cp_X\lesssim \cp_\gamma$ because ${\frac{d+1}d}\le 2$ and $\cp_\gamma$ is quadratic by Lemma \ref{entour}.

Clearly $\cp_X\ge \pr_X\ge\pr_\gamma$, 
and $\pr_X\lesssim \pr_\gamma$ because 
any path contributing to $\pr_X$ appears arbitrarily far from the origin of $\rho$.  The recurrence complexity of $X$ is at least linear because the function $\pr_\rho$ is increasing:
if $\pr_\rho(n+1)=\pr_\rho(n)$ for some $n$, some truncation $ \rho _k$ of $\rho$ satisfies $\cp_{\rho _k}(n+1)=\cp_{\rho _k}(n)$ hence is 
  ultimately periodic   by \cite{morse}, a contradiction.
  \end{proof}

\subsection{Recurrence complexity of paths}

\begin{prop} \label{comprec}
 The recurrence complexity $\pr_{\gamma}$ of any path $\gamma$  in $G$ is equivalent to  $1,n,n\log\log n, n\log n$, or $n^2$.
\end{prop}

We start by the following simple remark.

\begin{lem} \label{bete}
Assume that  $\gamma$ splits as $\gamma_1\cdot\gamma_2$. If $\pr_{\gamma_1}\gtrsim n$ or $\pr_{\gamma_2} \gtrsim n$, then $\pr_{\gamma}$ is the maximum of $\pr_{\gamma_1}$, $\pr_{\gamma_2}$. 

\end{lem}

\begin{proof}
 This follows from stabilisation of prefixes (Lemma \ref{stabpref}).  Words of length $n$ appearing in $f^k_\sharp(\gamma)$ appear in $f^k_\sharp(\gamma_1)$ or in $f^k_\sharp(\gamma_2)$ or contain the image of the splitting point. By Lemma \ref{stabpref}, neglecting the third type causes a linear error in $\pr_{\gamma}$.
\end{proof}

\begin{proof}[Proof of Proposition \ref{comprec}]
We first consider the recurrence complexity of terms   other than connecting paths.
  Fixed edges, linear edges, INP's and exceptional paths have bounded recurrence complexity.
Now let $e$ be a   non-linear and non-fixed edge  in an EG or NEG stratum. Since $e$ appears as a term in infinitely many images $f^k_\sharp(e)$, its recurrence complexity $\pr_{\gamma}$ is equal to its total complexity $\cp_\gamma$, which   is unbounded and satisfies one of the conclusions of the proposition by 
Proposition \ref {comptot}.

Now suppose that $\gamma$ is completely split, and no term is a connecting path. If every term of $\gamma$ has bounded recurrence complexity,  
one easily checks that $\pr_\gamma$  is bounded or linear.   Otherwise, using Lemma \ref{bete}, one shows by induction on the number of terms that the proposition holds for $\gamma$.

  This proves the proposition when no connecting path occurs, since any $\gamma$ has a completely split image $f^k_\sharp(\gamma)$.
To handle connecting paths, one shows by induction on height that connecting paths satisfy the proposition, and moreover any connecting path $\gamma$ with bounded recurrence complexity has an image $f^k_\sharp(\gamma)$ whose complete splitting only contains fixed edges, linear edges, INP's and exceptional paths. One then argues as before.
\end{proof}

\subsection{Recurrence complexity of fixed points}

\begin{thm} \label{compar3}
Let $\alpha$ be an automorphism of a finitely generated free group $F$, and let $X\in\bo F$ be fixed by $\alpha$. Assume $X\notin\bo\Fix\alpha$. 

\begin{itemize}
\item The recurrence complexity $\pr_X$ is equivalent to $n,n\log\log n, n\log n$, or $n^2$.
\item If the usual complexity   $\cp_X$ is not quadratic, in particular if $\alpha$ is fully irreducible or atoroidal, then 
$  \pr_X\sim \cp_X$. 
\end{itemize}
\end{thm}

 The proof requires the following lemma.

\begin{lem} [Nielsen paths are short] \label{long}
  Let $f:G\to G$ be a completely split train track map.
Let $\gamma$ be a completely split path such that no  $f^k_\sharp(\gamma)$ has a term which  is a linear edge or an exceptional path. There is a number $C$ such that, if   $\delta$ is any non-growing full subpath of some $f^k_\sharp(\gamma)$, then $\delta$ has length   $ | \delta | \le C$. 
\end{lem}

\begin{proof} 
   Assuming that the lemma is wrong, fix a counterexample  $\gamma$ of minimum height $r$. Thus the paths $f^k_\sharp(\gamma) $ contain arbitrarily long non-growing full subpaths. Write $\gamma$ as a concatenation of growing terms and maximal non-growing full subpaths. The length of these non-growing paths remains bounded when applying $f^k_\sharp $, so some growing term has to contradict the lemma. This term cannot be an INP, an exceptional path, a connecting path (by minimality of $r$), so it is an edge  $e$ in an NEG or EG stratum $H_r$. We show that both possibilities lead to a contradiction.

If it is an NEG edge $e$ with $f(e)=eu$, there is a bound for the length of all non-growing full subpaths of the paths $f^k_\sharp(u) $, because otherwise $u$ would be a counterexample of   height $<r$.  Since $e$ is a counterexample,   some $f^k_\sharp(u) $ has to be non-growing. This means that $u$ is pre-Nielsen and $e$ is linear,  contradicting the assumptions of the lemma.

In the EG case, consider all maximal full subpaths $\theta_i$ of   height $<r$ in $f(e')$, for all edges $e'$  of height $r$. 
 By minimality of $r$, there is a bound $C$ for the length of all non-growing full subpaths of the paths $f^k_\sharp(\theta_i)$, for $k\ge0$. We claim that $C$ bounds the length of non-growing full subpaths of $f^k_\sharp(e) $. 
 
 Indeed, let $\delta$ be a non-growing full subpath of $f^k_\sharp(e)$, of length $>C$. Consider the largest $q\le k-1$ such that $\delta$ is contained in $f^{k-q}_\sharp(e')$, with $e'$ an edge of height $r$ contained in $f^q_\sharp(e) $. Recall that the image of any edge in $H_r$ starts and ends with an edge in $H_r$, while no term of $\delta$ may be an edge in $H_r$.  This implies that 
 $\delta$ has to be contained in the image by $f^{k-q-1}_\sharp$ of a subpath of height $<r$ contained in $f(e')$.  
 It follows that $\delta$ has length $\le C$, a contradiction.
  \end{proof}

\begin{proof}[Proof of Theorem \ref{compar3}]
Represent $X$ by 
$\rho=\sigma \cdot f_\sharp (\sigma)\cdot f^2_\sharp(\sigma)\cdot f^3_\sharp(\sigma)\cdots$ 
as in Lemma \ref{morceaux}. 
The first assertion follows directly from Proposition \ref{comprec} and Lemma \ref{compar}. We prove the second one.

Each $f^k_\sharp(\sigma)$ has a complete splitting. Since the complexity is not quadratic, there is no linear edge or exceptional path. All growing terms grow exponentially and there is a bound $K$ for the length of growing terms and the length of non-growing full subpaths  by Lemma \ref{long}. By merging terms, we can construct a splitting of each $f^k_\sharp(\sigma)$, hence of $\rho$,  such that all pieces are growing and have length at most $3K$. 

Consider all   pieces $b,b'$ such that $bb'$   appears in this splitting of $\rho$. We truncate $\rho$, by removing some initial path 
$ \sigma \cdot f_\sharp (\sigma)\cdot {\dots} \cdot f^i _\sharp(\sigma)$
so  that, if $bb'$ appears, it appears infinitely often. We   write the splitting of $ f^{i +1}_\sharp(\sigma) f^{i+2} _{\sharp}(\sigma)$ by pieces as 
$b_1\cdot {\dots} \cdot b_r$. 

As $j$ and $k$ vary, the paths $f^k_\sharp(b_jb_{j+1})$ cover (the truncated) $\rho$ with overlaps so, as in the proof of Lemma \ref{compar}, exponential growth implies that we can compute the usual complexity of $X$, up to a linear error, by considering only subpaths of $\rho$ contained in some $f^k_\sharp(b_jb_{j+1})$. Since every $b_jb_{j+1}$ appears infinitely often in  the splitting of $\rho$ constructed above, we conclude $  \pr_X\sim \cp_X$. 
 \end{proof}

\subsection{Complexity of laminations} \label{lami}

Laminations are an important tool in the theory of automorphisms of free groups. We refer to \cite{CHL1} for a thorough discussion of   various viewpoints: algebraic laminations, symbolic laminations, laminary languages. 

In this subsection we explain that any lamination has a complexity function, well-defined up to equivalence. Moreover, any $X\in\bo F$ generates a lamination $L_X$.  If $X$ is fixed by some automorphism $\alpha\in\Aut(F)$, we use Theorem  
\ref{compar3} to control the complexity of $L_X$. This applies in particular to the attracting laminations defined in \cite{Tits1},    see   Corollary \ref {corlam} below.

To define a lamination simply, we fix a free basis $A$ of $F$. A (symbolic) lamination is a non-empty set $L$ of reduced bi-infinite words $z=(z_i)_{i\in\Z}$ in the alphabet $E=A\cup A^{-1}$ (called the leaves of the lamination)  which is closed (in the product topology on $E^\Z$), invariant under the shift $\sigma$ (defined by  $(\sigma(z))_i=z_{i+1}$), and invariant under the flip $\tau$ defined by $(\tau(z))_i=(z_{-i})^{\m}$.

As in Subsection \ref{cfg}, there is an action of $\Aut(F)$ on the set of laminations: replace each letter in $E$ by its image, and reduce   (bounded backtracking guarantees that this is well-defined, see \cite{CHL1});  inner automorphisms act trivially, so this is really an action of $\Out(F)$. Similarly, one may express a lamination $L$ in a different basis $B$: replace each letter in $E$ by its expression as a word on $B^{\pm1}$ and reduce.

The complexity $\cp_L$ of a lamination $L$ is defined by counting all the subwords appearing in some leaf of $L$. Arguing as in Subsection \ref{cfg}, one sees that $\cp_L$ is well-defined up to equivalence.

Now let $X$ be an element of $\bo F$, viewed as a right-infinite word $X_A$ on the alphabet $E$. One defines the language $\call_X$ associated to $X$ as the set of all finite words $u$ such that $u$ or $u\m$ appears infinitely often as a subword of $X_A$   (note that $\call_X$ does not change if we truncate $X$). It is a \emph{laminary language} in the sense of \cite{CHL1}. One then defines the lamination $L_X$ associated to $X$ as the set of words $z=(z_i)_{i\in\Z}$ as above such that any finite subword of $z$ belongs to $\call_X$. The assignment $X\mapsto L_X$  is equivariant with respect to the action of $\Aut(F)$.

It is clear from this construction that, given $X\in\bo F$, the complexity of $L_X$ is equivalent to the recurrence complexity of $X$. We may therefore use Theorem \ref{compar3}.

\begin{thm} \label{comparlam}
Let $\alpha$ be an automorphism of a finitely generated free group $F$, and let $X\in\bo F$ be fixed by $\alpha$. Assume $X\notin\bo\Fix\alpha$. Let $L_X$ be the lamination associated to $X$.

\begin{itemize}
\item The   complexity of $L_X$ is equivalent to $n,n\log\log n, n\log n$, or $n^2$.
\item If the usual complexity   $\cp_X$ of $X$ is not quadratic, in particular if $\alpha$ is fully irreducible or atoroidal, then the complexity of $L_X$ is equivalent to $\cp_X$. \qed
\end{itemize}
\end{thm}

We conclude by giving examples illustrating this theorem. 
Recall (Example \ref{rcdif}) that the complexity of $L_X$ may differ from that of $\cp_X$.

\begin{example}[Attracting laminations] \label{atlam}  Let $f:G\to G$ be a completely split train track map representing $\Phi\in\Out(F)$ as in Subsection \ref{cstt}. 
Let $H_r$ be an exponential stratum. Replacing $f$ by a power if needed, we can find an edge $e$ in $H_r$ such that $f(e)$ splits as $eue_2$ with $e_2$ an edge in $H_r$. The union of the increasing sequence $e\inc f(e)\inc f^2_\sharp(e)\cdots \inc f^p_\sharp(e)\inc \cdots$ is a ray $\rho$ in $G$ which defines an element $X\in\bo F$ (well-defined up to the action of $F$ on $\bo F$ by left-multiplication). This element $X$ is an attracting fixed point of some representative $\alpha$ of $\Phi$ in $\Aut(F)$, and the lamination $L_X$ associated to $X$ is the attracting lamination associated to the stratum $H_r$ (see \cite{gafa}, \cite{Tits1}). 
Similarly, if $e$ is an edge in an NEG stratum with $f(e)=eu$, and $u$ is growing, then   $X=euf_\sharp(u)f^2_\sharp(u)\cdots$
 is an attracting fixed point of a representative $\alpha$.  
 
 \end{example}

 \begin{cor}\label{corlam}  Let $\Phi\in\Out(\F_n)$.
 The complexity of any attracting lamination of   $\Phi$ is linear, quadratic, equivalent to $n\log n$, or equivalent to $n\log \log n$. It is linear if $\Phi$ is fully irreducible, at most $n\log n$ if $\Phi$ is atoroidal. \qed
\end{cor}
 
 The representatives $\alpha$ constructed above are rather special. ``Most'' representatives $\alpha$ of $\Phi$ have exactly one attracting fixed point    (see \cite{most}), but it is not given by the previous construction.

\begin{example}\label{lamiex}
 Let $F$ be the free group on two generators $a,b$. Let $\Phi$ be the outer automorphism of $F$ represented by $\alpha$ sending $a$ to $ab$ and $b$ to $bab$ (it is realized geometrically by a pseudo-Anosov homeomorphism of a punctured torus and has one EG stratum). 
 The infinite word $X=ab^2ab^2abab\cdots=\lim_{n\to\infty}\alpha^n(a)$ is an attracting fixed point and the associated lamination is the attracting lamination of $\Phi$.
 
 Let   $\gamma=[a,b]$;  it is fixed by $\alpha$.   For $k\ge1$, define $\alpha_k$ by $\alpha_k(g)=a\gamma^k\alpha(g)\gamma^{-k} a\m$. It is another representative of $\Phi$, with attracting fixed point $X_k=a\gamma^k\alpha(a)\gamma^k \alpha^2(a)\cdots$. The laminary language $\call_{X_k}$ consists of  
 all words $w$ such that $w^{\pm1}$ is contained in some $\alpha^n(a)\gamma^k\alpha^{n+1}(a)$. 
\end{example}
 
 This example shows that, for a given $\Phi$, Theorem \ref{comparlam} may apply to infinitely many different laminations.

%\bibliography{bibliocomplexite2022reecrit} \bibliographystyle{plain}

\bigskip\bigskip
\small
{\begin{flushleft}
Arnaud Hilion\\
Institut de Math\'ematiques de Toulouse UMR 5219. Universit\'e de Toulouse \& CNRS. UPS, F-31062 Toulouse Cedex 9, France\\
\emph{e-mail: }\texttt{arnaud.hilion@math.univ-toulouse.fr}\\[8mm]

Gilbert Levitt\\
Laboratoire de Math\'ematiques Nicolas Oresme (LMNO)\\
Universit\'e de Caen et CNRS (UMR 6139)\\
(Pour Shanghai : Normandie Univ, UNICAEN, CNRS, LMNO, 14000 Caen, France)\\
 \emph{e-mail: }\texttt{levitt@unicaen.fr}\\[8mm]

\end{flushleft}
}

 \end{document}